\documentclass[12pt]{amsart}

\usepackage{latexsym,fancyhdr,amssymb,color,amsmath,amsthm,graphicx,listings,comment}
\usepackage[section]{placeins} 
\newtheorem{thm}{Theorem} \newtheorem{lemma}{Lemma} \newtheorem{propo}{Proposition} 
\newtheorem{coro}{Corollary} 
\let\paragraph\subsection
\newcommand{\NN}{\mathbb{N}} \newcommand{\ZZ}{\mathbb{Z}} \newcommand{\DD}{\mathbb{D}}
  \newcommand{\TT}{\mathbb{T}}
\def\G{\mathcal{G}}
\def\C{\mathcal{C}}
\def\osquare{\mbox{\ooalign{$\times$\cr\hidewidth$\square$\hidewidth\cr}} }

\title{The strong ring of simplicial complexes}
\author{Oliver Knill}
\date{Aug 5, 2017}
\address{Department of Mathematics \\ Harvard University \\ Cambridge, MA, 02138 }
\subjclass{05C99, 13F55, 55U10, 68R05}
\keywords{Topological arithmetic, Zykov ring, Sabidussi ring, Mass gap}

\begin{document}
\maketitle

\begin{abstract}
The strong ring $R$ is a commutative ring generated by finite abstract simplicial complexes.
To every $G \in R$ belongs a Hodge Laplacian $H=H(G)=D^2=(d+d^*)^2$ determining the cohomology and 
a unimodular connection operator $L=L(G)$. The sum of the matrix 
entries of $g=L^{-1}$ is the Euler characteristic $\chi(G)$. For any $A,B \in R$ the spectra of $H$ satisfy 
$\sigma(H(A \times B)) = \sigma(H(A)) + \sigma(H(B))$ and
the spectra of $L$ satisfy $\sigma(L(A \times B)) = \sigma(L(A)) \cdot \sigma(L(B))$ as 
$L(A \times B) = L(A) \otimes L(B)$ is the matrix tensor product. The inductive dimension of $A \times B$ is
the sum of the inductive dimension of $A$ and $B$. The dimensions of the kernels of the form Laplacians $H_k(G)$ 
in $H(G)$ are the Betti numbers $b_k(G)$ but as the additive disjoint union monoid
is extended to a group, they are now signed with $b_k(-G)=-b_k(G)$. 
The maps assigning to $G$ its Poincar\'e polynomial $p_G(t)=\sum_{k=0} b_k(G) t^k$ or
Euler polynomials $e_G(t)=\sum_{k=0} v_k(G) t^k$ are ring homomorphisms from $R$ to $\ZZ[t]$.
Also $G \to \chi(G)=p(-1)=e(-1) \in \ZZ$ is a ring homomorphism.
Kuenneth for cohomology groups $H^k(G)$ is explicit via Hodge: a basis for $H^k(A \times B)$ is obtained from a basis of the factors.
The product in $R$ produces the strong product for the connection graphs. These relations generalize to 
Wu characteristic. $R$ is a subring of the full Stanley-Reisner ring $S$, a 
subring of a quotient ring of the polynomial ring $Z[x_1,x_2, \dots ]$. An object $G \in R$ can 
be visualized by ts Barycentric refinement $G_1$ and its connection graph $G'$. 
Theorems like Gauss-Bonnet, Poincar\'e-Hopf or Brouwer-Lefschetz for Euler and Wu characteristic 
extend to the strong ring. The isomorphism $G \to G'$ to a subring of the strong Sabidussi ring shows that
the multiplicative primes in $R$ are the simplicial complexes and
that connected elements in $R$ have a unique prime factorization.
The Sabidussi ring is dual to the Zykov ring, in which the Zykov join is the addition, which is a 
sphere-preserving operation. The Barycentric limit theorem implies that the connection Laplacian of the 
lattice $Z^d$ remains invertible in the infinite volume limit: there is a mass gap containing 
$[-1/5^d,1/5^d]$ for any dimension $d$. 
\end{abstract}

\section{Energy theorem}

\paragraph{}
A {\bf finite abstract simplicial complex} is a finite set of 
non-empty sets which is closed under the operation of taking finite non-empty 
subsets. The {\bf connection graph} $G'$ of such a {\bf simplicial complex} 
$G$ has as the vertices the sets of $G$ and as the edge set the pairs of sets which intersect.
Given two simplicial complexes $G$ and $K$, their {\bf sum} $G \oplus K$ is the disjoint union and 
the {\bf Cartesian product} $G \times K$ is the set of all set Cartesian products 
$x \times y$, where $x \in G, y \in K$. While $G \times K$ is no
simplicial complex any more if both factors are different from the one-point complex $K_1$, 
it still has a connection graph $(G \times K)'$, the graph for which the vertices are the sets $x \times y$ 
and where two sets are connected if they intersect. 
The {\bf Barycentric refinement} $(G \times K)_1$ of $G \times K$ is the Whitney
complex of the graph with the same vertex set as $G'$ but where two sets are 
connected if and only if one is contained in the other. 
The matrix $L=L(G)=1+A$, where $A$ is the adjacency
matrix of $G'$ is the {\bf connection Laplacian} of $G$. 

\paragraph{}
All geometric objects considered here are finite and combinatorial and 
all operators are finite matrices. Only in the
last section when we look at the Barycentric limit of the discrete lattice $Z^d$, the operators become
almost periodic on a profinite group, but there will be universal bounds on the norm of the inverse. 
One of the goals of his note is to extend the following theorem to the strong ring generated by 
$G_{i_1} \times \cdots \times G_{i_k}$. 

\begin{thm}[Unimodularity theorem]
For every simplicial complex $G$, the connection Laplacian $L(G)$ is unimodular. 
\end{thm}
We have proven this in \cite{Unimodularity} inductively by building up 
the simplicial complex $G$ as a discrete CW-complex starting with the 
zero-dimensional skeleton, then adding one-dimensional cells, 
reaching the one-dimensional skeleton of $G$, then adding triangles etc. 
continuing until the entire complex is built up.
In every step, if a cell $x$ is added, this means that a ball $B(x)$ is 
glued in along a sphere $S(x)$, the Fredholm determinant 
${\rm det}(L)={\rm det}(1+A)$ of the adjacency matrix $A$ of $G'$ 
is multiplied by 
$\omega(x) = 1-\chi(S(x)) =(-1)^{{\rm dim}(x)} \in \{ -1,1\}$,
where $S(x)$ is the unit sphere in the Barycentric refinement $G_1$ of $G$,
which is the Whitney complex of a graph. 
The determinant of $L(G)$ is now equal to the {\bf Fermi characteristic}
$\prod_x \omega(x) \in \{-1,1\}$, a multiplicative analogue of the 
{\bf Euler characteristic} $\sum_x \omega(x)$ of $G$. Having determinant $1$
or $-1$, the matrix is unimodular.

\paragraph{}
As a consequence of the unimodularity theorem, 
the inverse $g=L^{-1}$ of $L$ produces {\bf Green functions} in the form of 
integer entries $g(x,y)$ which can be seen as the {\bf potential energy}
between the two simplices $x,y$. The sum $V(x)=\sum_y g(x,y)$
is the {\bf potential} at $x$ as it adds up the potential energy $g(x,y)$ induced from the
other simplices. It can also be interpreted as a {\bf curvature} because the
Euler characteristic $\chi(G) = \sum_{x \in G} \omega(x)$ is the sum over
the Green function entries: 

\begin{thm}[Energy theorem]
For every simplicial complex $G$, we have $\sum_x V(x) = \sum_{x,y} g(x,y) = \chi(G)$. 
\end{thm}

\paragraph{}
This formula is a Gauss-Bonnet formula, when the row sum is interpreted as a curvature. 
It can also be seen as a {\bf Poincar\'e-Hopf formula}
because $V(x)=(-1)^{{\rm dim}(x)} (1-\chi(S(x)) = (1-\chi(S^-_f(x)))$
is a {\bf Poincar\'e-Hopf} index for the {\bf Morse function} $f(x)=-{\rm dim}(x)$ 
on the Barycentric refinement $G_1$ of $G$. A function on a graph is a {\bf Morse functional} if
every $S^-_f(x)$ is a discrete sphere, where $S^-_f(x)=\{ y \in S(x) \; | \; f(y)<f(x) \; \}$.
The proof reduces the energy theorem to the Poincar\'e-Hopf formula which is 
dual to the definition of Euler characteristic: the Poincar\'e-Hopf index of $-f={\rm dim}$
is $\omega(x)$ and $\sum_x \omega(x)$ is the definition of Euler characteristic as we sum
over simplices. In the Barycentric refinement, where we sum over vertices then $\omega(x)$ can be
seen as a curvature. In order to reduce the energy theorem to Poincar\'e-Hopf, one has to
show that $V(x) = \sum_y g(x,y)$ agrees with $(-1)^{{\rm dim}(x)} g(x,x)$. See also 
\cite{Spheregeometry} for an interpretation of the diagonal elements. 

\paragraph{}
The {\bf Barycentric refined complex} $G_1$ is a set of subsets of the power set 
$2^{G}$ of $G$.  It consists of set of sets in $G$ which pairwise are contained 
in each other. It is the {\bf Whitney complex} 
of the graph $G_1=(V,E)$, where $V=G$ and $E$ is the set of $(a,b)$ with 
$a \subset b$ or $b \subset a$. Since $G_1$ is the Whitney complex of 
a graph, we can then use a more intuitive picture of the complex, as graphs
can be drawn and visualized and are hardwired already as structures in 
computer algebra systems. Also definitions become easier. To define the
inductive dimension for example, its convenient to define it for
simplicial complexes as the corresponding number for the Barycentric refined
complex, where one deals with the Whitney complex of a graph. 

\section{The Sabidussi ring}

\paragraph{}
The {\bf strong product} of two finite simple graphs $G=(V,E)$ and $H=(W,F)$ 
is the graph $G \osquare H = (V \times W, \{ ((a,b),(c,d)) \; a=c, (b,d) \in F \} 
                                             \cup  \{ ((a,b),(c,d)) \; b=d, (a,c) \in E \}
                                             \cup  \} ((a,b),(c,d)) \; (a,c) \in E {\rm and} (b,d) \in F \} )$.
It is an associative product introduced by Sabidussi \cite{Sabidussi}. Together with the 
disjoint union $\oplus$ as addition on signed complexes, 
it defines the {\bf strong Sabidussi ring of graphs}. 
We have started to look at the arithmetic in \cite{ArithmeticGraphs}. 
We will relate it in a moment to the strong ring
of simpicial complexes. In the Sabidussi ring of graphs, 
the additive monoid $(\G_0,\oplus)$ of finite simple graphs with 
zero element $(\emptyset,\emptyset)$ is first extended to a larger class $(\G,\oplus)$ 
which is a Grothendieck group. The elements of this group are
then {\bf signed graphs}, where each connected component can have either a positive or 
negative sign. The {\bf additive primes} in the strong ring are the connected components. 
An element in the group, in which both additive primes $A$ and $-A$ appear, is equivalent 
to a signed graph in which both components are deleted. 
The {\bf complement} of a graph $G=(V,E)$ is denoted by
$\overline{G}=(V,\overline{E})$, where $\overline{E}$ is the set of pairs $(a,b)$
not in $E$. 

\paragraph{}
The {\bf join}  $G + H = (V \cup W, E \cup F \cup \{ (a,b), a \in V, b \in W \})$ is an 
addition introduced by Zykov \cite{Zykov}. It is dual to the disjoint union 
$G + H = \overline{ \overline{G} \oplus \overline{H}}$. The {\bf large product} 
$G \star H = (V \times W, \{ (a,b),(c,d) \; (a,c) \in E \; {\rm or} \; (b,d) \in E \})$ 
is dual to the strong product. 
In other words, the dual to the strong ring $(\G,\oplus,\osquare,0,1)$ is the
{\bf large ring} $(\G,+,\star,0,1)$. Because the complement operation $G \to \overline{G}$ is 
invertible and compatible with the ring operations, the two rings are isomorphic.
The additive primes in the strong ring are the connected sets. 
Sabidussi has shown that every connected set has a unique multiplicative prime factorization. 
It follows that the strong ring of graphs is an integral domain. 
It is not a unique factorization domain however. There are disconnected graphs
which can be written in two different ways as a product of two graphs. 

\paragraph{}
The Sabidussi and Zykov operations could also be defined for simplicial complexes but we don't need
this as it is better to look at the Cartesian product and relate it to the strong
product of connection graphs. But here is a definition:
the disjoint union $G \oplus H$ is just $G \cup H$ assuming that the simplices are disjoint.
The Zykov sum, or join is $G + H = G \cup H \cup \{ x \cup y \; | \; x \in G, y \in H \}$. 
If $\pi_k$ denote the projections from the 
{\bf set theoretical Cartesian product} $X \times Y$ to $X$ or $Y$, 
the Zykov product can be defined as
$G \star  H = \{ A  \subset X \times Y \; | \; \pi_1(A) \in G \; {\rm or} \;  \pi_2(A) \in H \}$.
It can be written as $G \star H = G \times 2^{V(H)} \cup 2^{V(G)} \times H$, 
where $2^X$ is the {\bf power set} of $X$ and $V(G)= \bigcup_{A \in G}$. 

\section{The Stanley-Reisner ring} 

\paragraph{}
The {\bf Stanley-Reisner ring} $S$ is a subring $\bigcup_n \ZZ[x_1,x_2, \dots x_n]/I_n$, where 
$I_n$ is the ideal generated $x_i^2$ and $f_G-f_H$, where $f_G, f_H$ are ring elements representing
the isomorphic finite complexes $G,H$. The ring $S$ contains all elements for which $f(0,0,0,... 0) = 0$. 
It is a quotient ring of the 
{\bf ring of chains} $C \subset \bigcup_n \ZZ[x_1,x_2, \dots x_n]/J_n$ with $J_n$ is the ideal generated by 
squares. A signed version of ring of chains is used in algebraic topology. It is the {\bf free Abelian group}
generated by finite simplices represented by monoids in the ring.
The ring of chains $C$ is larger than the Stanley-Reisner ring as for example, the ring element $f=x-y$ is zero in $S$.
The Stanley-Reisner ring is helpful as
every simplicial complex $G$ and especially every Whitney complex of a graph can be described 
algebraically with a polyonomial: if $V= \cup_{A \in G} A =\{x_1, \dots, x_n\}$ is the base
set of the finite abstract simplicial complex $G$, define the monoid 
$x_A = \prod_{x \in A} x$ and then $f_G=\sum_{A \in G} x_A$. We initially have computed the product using
this algebraic picture in \cite{KnillKuenneth}. 

\paragraph{}
The circular graph $G=C_4$ for example gives
$f_G = x+y+z+w+xy+yz+zw+wx$ and the triangle $H=K_3$ is described by $f_H= a+b+c+ab+ac+bc+abc$. 
The Stanley-Reisner ring also contains elements like $7xy-3x+5y$ which are only chains and not simplicial complexes.
The addition $f_G + f_H$ is the disjoint union of the two complexes, where
different variables are used when adding two complexes. So, for $G=K_2=x+y+xy$ we have 
$G+G = x+y+xy+a+b+ab$ and $G+G+G=x+y+xy+a+b+ab+u+v+uv$. The negative complex $-G$ is $-x-y-xy$. 
The complex $G-G = -x-y-xy+a+b+ab$ is in the ideal divided out so that it becomes $0$ in the
quotient ring. The Stanley-Reisner ring is large. It contains elements like $xyz$ which can
not be represented as linear combinations $\sum_i a_i G_i$ of simplicial complexes $G_i$. 
But we like to see the strong ring $R$ embedded in the full Stanley-Reisner ring $S$. 

\paragraph{}
If $G$ and $H$ are simplicial complexes represented by polynomials $f_g,f_H$, then
the product $f_G f_H = f_{G \times H}$ is not a 
simplicial complex any more in general. Take $f_{K_2} f_{K_2} = (a+b+a b)(c+d+c d)$ for example which is
$ a c + b c + a b c + a d + b d + a b d + a c d + b c d + a b c d$. We can not interpret $ac$ as a new 
single variable as $bc$ is also there and their intersection $ac \cap bc = c$ could not be represented
in the complex. A simplicial complex is by definition closed for non-empty intersections as such an
intersection is a subset of both sets. We can still form the subring S in R generated by the 
simplicial complexes and call it the {\bf strong ring}. The 
reason for the name is that on the connection graph level it leads to the strong product of graphs. 
The strong ring of simplicial complexes will be isomorphic to a subring of the 
Sabidussi ring of graphs. 

\paragraph{}
The Stanley-Reisner picture allows for a concrete implementation of the additive 
Grothendieck group which extends the monoid given by the disjoint union as addition. 
The Stanley-Reisner ring is usually a ring attached to a single geometric object. 
The full Stanley-Reisner ring allows to represent any element in the strong ring. 
It is however too large for many concepts in combinatorial topology, where
finite dimensional rings are used to describe a simple complex. 
One can not see each individual element in $S$ as a geometric object on its own, as cohomology and
the unimodularity theorem fail on such an object. The chain $f=xy+yz+y$ for example is no simplicial complex.
Its connection graph is a complete graph for which the Fredholm determinant ${\rm det}(1+A)$ is zero.
Its boundary is not a subset of $f$, its connection graph is $K_3$ as all ingredient, the two edges and the 
single vertex all intersect with each other. The elements of the Stanley-Reisner picture behave like
measurable sets in a $\sigma$-algebra. One can attach to every $f$ in the full Stanley-Reiser ring an {\bf Euler characteristic}
$\chi(f) = -f(-1,-1, \dots, -1)$ which satisfies $\chi(f+g)=\chi(f)+\chi(g), \chi(f g) = \chi(f) \chi(g)$ 
but this in general does not have a cohomological analog via Euler-Poincar\'e. This only holds in the strong ring. 

\paragraph{}
While the Cartesian product $G \times H$ of two simplicial complexes is not a simplicial complex any more,
the product can be represented by an element $f_G f_H$ in the Stanley-Reisner ring. The
{\bf strong ring} is defined as the subring $S$ of the full Stanley-Reisner ring $R$ 
which is generated by simplicial complexes. 
Every element in the strong ring $S$ is a sum $\sum_I a_I f_{G_I}$, where  $a_I \in \ZZ$ and
for every finite subset $I \subset \NN$, the notation $f_{G_I} = \prod_{i \in I} f_{G_i}$ is used.
The ring of chain contains $\sum_I a_I x_I$, with $x_I=x_{i_1} \cdot x_{i_k}$, where 
$A_I =\{ x_{i_1}, \dots, x_{i_k} \}$ are finite subsets of $\{x_1,x_2, \dots \}$. 
Most of the elements are not simplicial complexes any more. The ring of chains has been 
used since the beginnings of combinatorial topology. But its elements can also be 
described by graphs, connection graphs. 

\paragraph{}
The strong ring has the empty complex as the zero element and the one-point complex $K_1$ as the
one element. The element $-K_1$ is the $-1$ element. The strong ring contains $\ZZ$ by identifying 
the zero dimensional complexes $P_n$ with $n$. It contains 
elements like $x+y+xy - (a+b+c+ab+bc+ac)$ which is a sum $K_2 + C_3$ 
of a Whitney complex $K_2$ a non-Whitney complex $C_3$. The triangle $K_3$ 
is represented by $(a+b+c+ab+bc+ac+abc)$. The strong ring does not contain elements like 
$x+y+xy - (y+z+yz)$ as in the later case, we don't have a linear combination of 
simplicial complexes as the sum is not a disjoint union. 
The element $G=x+y+xy - (a+b+ab)$ is identified with the zero element $0$ as $G=K_2 - K_2$
and $G=x+y+xy + (a+b+ab) = K_2 + K_2$ can be written as $2 K_2 = P_2 \times K_2$,
where $P_2=u+v$ is the zero-dimensional complex representing $2$. The
multiplication honors signs so that for example $G \times (-H) = -G \times H$ and
more generally $a (G \times H) = (a G) \times H) = G \times (a H)$ for any zero
dimensional signed complex $a = P_a$ which follows from the distributivity 
$H \times (G_1 + G_2) = H \times G_1 + H \times G_2$. 

\section{The connection lemma}

\paragraph{}
The following lemma shows that the {\bf strong ring of simplicial complexes}
is isomorphic to a subring of the {\bf Sabidussi ring}. 
First of all, we can extend the notion of {\bf connection graph} from simplicial complexes to 
products $G=G_{i_1} \times \cdots \times G_{i_n}$ of simplicial complexes and so to the 
strong ring. The vertex set of $G$ is the {\bf Cartesian product} $V'=V_1 \times \cdots V_n$ 
of the base sets $V_i=V(G_i) = \bigcup_{A \in G_i} A$. Two different elements in $V$ are connected 
in the connection graph $G'$ if they intersect as sets. 
This defines a finite simple graph $G'=(V',E')$. Also a multiple $\lambda G$ of a
complex is just mapped into the multiple $\lambda G'$ of the connection graph, if $\lambda$ 
is an integer. 

\begin{lemma}[Connection lemma]
$(G \times H)' = G' \osquare H'$. 
\end{lemma}
\begin{proof}
In $G \times H$, two simplices $(a \times b), (c \times d)$ in the vertex set 
$V = \{ (x,y) \; | \; x \in G, y \in H \}$ of $(G \times H)'$ are connected in 
$(G \times H)'$ if $(a \times b) \cap (c \times d)$ is not empty. 
But that means that either $a \cap c$ is not empty or then that $b \times d$ is 
not empty or then that $a=c$ and $b \cap d$ 
is not empty or then that $b=d$ and $a \cap c$ is not empty. 
\end{proof} 

\paragraph{}
The strong connection ring is a sub ring of the {\bf Sabidussi ring}
$(\G,\oplus,\osquare,0,1)$, in which objects are signed graphs. 
While the later contains all graphs, the former only contains ring elements of the form $G'$, 
where $G$ is a simplicial complex the {\bf strong connection ring}. The lemma allows us to avoid
seeing the elements as a subspace of abstract finite CW complexes for which the Cartesian 
product is problematic. Both the full Stanley-Reisner ring and the Sabidussi rings are too large.
The energy theorem does not hold in the full Stanley-Reisner ring, as the example $G=xy+yz+y$ 
shows, where $G'=K_3$ is the complete graph for which the Fredholm determinant is zero.
We would need to complete it to a simplicial complex like $G=xy+yz+x+y+z$ for which 
the connection Laplacian $L(G)$ has a non-zero Fredholm determinant.

\paragraph{}
Given a simplicial complex $G$, let $\sigma(G)$ denote the {\bf connection spectrum} of $G$. It is
the spectrum of the connection Laplacian $L(G)$. The trace ${\rm tr}(L(G))$ is the number of cells in the complex.
It is a measure for the {\bf total spectral energy} of the complex. Like the potential 
theoretic total energy $\chi(G)$, the connection spectrum and total energy are compatible with arithmetic.
The trace ${\rm tr}(G)$ is the total number of cells and can be written in the Stanley-Reisner
picture as $f_G(1,1,1, \dots , 1)$. 

\begin{coro}[Spectral compatibility]
$\sigma(G \times H) = \sigma(G) \sigma(H)$.
\end{coro}
\begin{proof}
It is a general fact that the Fredholm adjacency matrices tensor under
the strong ring multiplication. This implies that spectra multiply. 
\end{proof} 

The trace is therefore a ring homomorphism from the strong ring $S$ to the integers but
since the trace ${\rm tr}(L(G))$ is the number of cells, this is obvious.

\paragraph{}
We see that from a spectral point of view, it is good to look at the Fredholm 
adjacency matrix $1+A(L)$ of the connection graph. The operator $L(G)$ is an
operator on the same Hilbert space than the Hodge Laplacian $H$ of $G$ which is 
used to describe cohomology altebraically. We will look at the Hodge Laplacian later. 

\paragraph{}
As a consequence of the tensor property of the connection Laplacians, we also know
that both the unimodularity theorem as well as the energy theorem extend to the
strong connection ring.

\begin{coro}[Energy theorem for connection ring]
Every connection Laplacian of a strong ring element is unimodular and has
the property that the total energy is the Euler characteristic. 
\end{coro}
\begin{proof}
Linear algebra tells that if $L$ is a $n \times n$ matrix and $M$ is a $m \times m$ matrix, 
then $\det(L \otimes M) = \det(L)^m \det(M)^n$. If $L,M$ are connection
Laplacians, then $|\det(L)|=|\det(L)|=1$ and the product shares the property of having
determinant $1$ or $-1$. 
\end{proof}

\paragraph{}
In order to fix the additive part, we have to define the Fredholm determinant 
of $-G$, the negative complex to $G$. Since $\chi(-G)=-\chi(G)$, we have
$\omega(-x)=-\omega(x)$ for simplices and if we want to preserve the property
$\psi(G) = \prod_x \omega(x)$, we see that that defining $\psi(-G) = {\rm det}(-L(G))$ 
is the right thing. The {\bf connection Laplacian} of $-G$ is therefore defined as 
$-L(G)$. This also extends the energy theorem correctly. 
The sum over all matrix entries of the inverse of $L$ is the Euler characteristic.

\section{The Sabidussi theorem}

\paragraph{ } 
An element $G$ in the strong ring $S$ is called an {\bf additive prime} if it can 
not be decomposed as $G=G_1 \oplus G_2$ with both $G_i$ being non-empty. The
additive prime factorization is braking $G$ into 
{\bf connected components}.
A {\bf multiplicative prime} in the strong ring is an element which can not be
written as $G = G_1  \times G_2$, where both $G_i$ are not the one-element $K_1$. 

\begin{thm}[Sabidussi theorem]
Every additive prime in the Sabidusi ring has a unique multiplicative
prime factorization. 
\end{thm}

See \cite{Sabidussi}. See also \cite{ImrichKlavzar,HammackImrichKlavzar}
where also counter examples appear, if the connectivity assumption is dropped. 
The reason for the non-uniqueness is that $\NN[x]$ has no unique 
prime factorization: $(1+x+x^2)(1+x^3) =(1+x^2+x^4)(1+x)$.

\paragraph{}
Can there be primes in the strong ring that are not primes in the Sabidussi ring?
The factors need not necessarily have to be simplicial complexes. 
Is it possible that a simplicial complex can be factored into smaller 
components in the Stanley-Reisner ring? The answer is no, because a simplicial complex $G$
is described by a polynomial $f_G$ has linear parts. A product does not have linear parts.
We therefore also have a unique prime factorization for connected components in the strong ring. 
The Sabidussi theorem goes over to the strong ring. 

\begin{coro}
Every additive prime in the strong ring has a unique multiplicative 
prime factorization. The connected multiplicative primes in the strong ring 
are the connected simplicial complexes. 
\end{coro}

\paragraph{}
If we think of an element $G$ in the ring as a {\bf particle} and of $G \cup H$ 
as a {\bf pair of particles}, then the {\bf total spectral energy} is 
the sum as the eigenvalues adds up. As the individual $L$ spectra 
multiply, this provokes comparisons with the 
{\bf Fock space} of particles: when taking the 
disjoint union of spaces is that the Hilbert space of the particles is
the product space. If we look at the product of two spaces, then the 
Hilbert space is the tensor product. In some sense,
"particles" are elements in the strong ring. They are generated  by
prime pieces of "space". These are the elements in that ring 
which belong to simplicial complexes. 

\paragraph{}
Whether this picture painting particles as objects generated by space has
merit as a model in physics is here not important. Our point of view 
is purely mathematical: the geometric ``parts" have {\bf topological 
properties} like cohomology or energy or {\bf spectral properties} like 
spectral energy attached to them. 
More importantly, the geometric objects are elements in a 
ring for which the algebraic operations are compatible with the
topological properties. The additive primes in the ring are the connected components, the 
``particles" where each is composed of smaller parts thanks to a
unique prime factorization $G_1 \times \cdots \times G_n$ into {\bf multiplicative primes}. 
These {\bf elementary parts or particles} are just the connected {\bf simplicial complexes}.

\section{Simplicial cohomology} 

\paragraph{}
The {\bf Whitney complex} of a finite simple graph $G=(V,E)$ is the simplicial complex
in which the sets are the vertex sets of complete subgraphs of $G$. If $d$ denotes the 
exterior derivative of the Whitney complex, then the {\bf Hodge Laplacian} $H=(d+d^*)^2$ decomposes
into blocks $H_k(G)$ for which $b_k(G)= {\rm dim}({\rm ker}(H_k))$ are the {\bf Betti numbers}, the 
dimensions of the {\bf cohomology groups} $H^k(G) = {\rm ker}(d_k)/{\rm im}(d_{k-1}))$. The
{\bf Poincar\'e-polynomial} of $G$ is defined as $p_G(x) = \sum_{k=0}^{\infty} b_k(G) x^k$. 
For every connected complex $G$, define $p_{-G}(x)=-p_G(x)$. Complexes can now have negative
Betti numbers. There are also non-empty complexes with $p_G(x)=0$ like $G=C_4-C_5$.

\paragraph{}
The exterior derivative on a signed simplicial complex $G$ is defined as $df(x) = f(\delta x)$,
where $\delta$ is the boundary operation on simplices. 
The indidence matrix can also be defined as $d(x,y) = 1$ if $x \subset y$ and the orientation 
matches.
Note that this depends on the choice of the orientation of the simplices as we do not require any compatibility. 
It is a choice of basis in the Hilbert space on which the Laplacian will work. 
The {\bf exterior derivative} $d$ of a product $G_1 \times G_2$ of two complexes each
having the boundary operation $\delta_i$ is then given as 
$$  df(x,y) = f(\delta_1 x,y) +(-1)^{{\rm dim}(x)} f(x,\delta_2y) \; . $$
The {\bf Dirac operator} $D=d+d^*$ defines then the Hodge Laplacian $H=D^2$.
Both the connection Laplacian and Hodge Laplacian live on the same space. 

\paragraph{}
The Hodge theorem directly goes over from simplicial complexes to elements in the strong
ring. Let $H=\oplus_{k=0} H_k$ be the block diagonal decomposition of the
Hodge Laplacian and $G = \sum_{i=1}^n a_I G_I$ the additive decomposition into connected components
of products $G_I = G_{i_1} \times \cdots \times G_{i_n}$ of simplicial complexes. 
For every product $G_I$, we can write down a concrete exterior derivative $d_I$, Dirac operator $D(G_I)=d_I + d_I^*$ 
and Hodge Laplacian $H(G_I)= D(G_I)^2$. 

\begin{propo}[Hodge relation]
We have $b_k(G_I) = {\rm dim}({\rm ker}(H_k(G_I)))$. 
\end{propo}

The proof is the same as in the case of simplicial complexes as the cohomology of $G_I$ 
is based on a concrete exterior derivative. When adding connected components we
{\bf define} now
$$ b_k(\sum_I a_I G_I) = \sum_I a_I b_k(G_I) \; . $$

\paragraph{}
The strong ring has a remarkable compatibility with cohomology. In order to 
see the following result, one could use discrete homotopy notions or then refer to the
classical notions and call two complexes homotopic if their geometric realizations are 
homotopic. It is better however to stay in a combinatorial realm and ignore geometric
realizations.

\begin{thm}[Kuenneth]
The map $G \to p_G(x)$ is a ring homomorphism from the strong connection ring to $\ZZ[x]$. 
\end{thm} 
\begin{proof}
If $d_i$ are the exterior derivatives on $G_i$, we can write them as partial exterior
derivatives on the product space. We get from
$d f(x,y) = d_1 f(x,y) + (-1)^{{\rm dim}(x)} d_2(f(x,y)$ 
$$ d^* d f = d_1^* d_1 f + (-1)^{{\rm dim}(x)} d_1^* d_2 f 
           + (-1)^{{\rm dim}} d_2^* d_1 f + d_2^* d_2 f \; ,  $$
$$ d d^* f = d_1 d_1^* f + (-1)^{{\rm dim}(x)} d_1 d_2^* f 
           + (-1)^{{\rm dim}(x)} d_2 d_1^* f + d_2 d_2^* f \; . $$
Therefore $H f = H_1 f + H_2 f + (-1)^{{\rm dim}(x)} (d_1^*d_2 + d_1 d_2^* + d_2^* d_1 + d_2 d_1^*) f(x,y) )$.
Since Hodge gives an orthogonal decomposition 
$$  {\rm im}(d_i),{\rm im}(d_i^*), {\rm ker}(H_i) = {\rm ker}(d_i) \cap {\rm ker}(d_i^*) \; , $$
there is a basis in which $H(v,w) = (H(G_1)(v), H(G_2)(w))$. 
Every kernel element can be written as $(v,w)$,
where $v$ is in the kernel of $H_1$ and $w$ is in the kernel of $H_2$. 
\end{proof}

The Kuenneth formula follows also from \cite{KnillKuenneth} 
because the product is $(G \times H)_1$ and
the cohomology of the Barycentric refinement is the same.
It follows that the {\bf Euler-Poincar\'e formula} holds in general for elements in the
ring: the cohomological Euler characteristic $\sum_{k=0}^{\infty} b_k(G) (-1)^k$ is equal to 
the combinatorial Euler characteristic $\sum_{k=0}^{\infty} v_k(G) (-1)^k$, where
$(v_0,v_1, \dots)$ is the $f$-vector of $G$.  \\

There is also a cohomolog for the higher Wu characteristic $\omega_k(G)$. 

\section{Gauss-Bonnet, Poincar\'e-Hopf}

\paragraph{}
The definition $\sum_x \omega(x) = \chi(G)$ of Euler characteristic, with curvature $\omega(x) = (-1)^{{\rm dim}(x)}$
can be interpreted as a {\bf Gauss-Bonnet} result in $G_1$, the Barycentric refinement of a simplicial complex $G$. 
If $G$ is the Whitney complex of a graph, then we have a small set of vertices $V$, the zero dimensional simplices
in $G$. Pushing the curvature from the simplices to the vertices $v$, then produces the curvature 
$$ K(v) = \sum_{k=0}^{\infty} \frac{(-1)^k V_{k-1}}{(k+1)} 
        = 1 - \frac{V_0}{2} + \frac{V_1}{3} - \frac{V_2}{4} + \cdots  \; , $$
where $V_k(v)$ is the number of $k$-dimensional simplices containing $v$ and $V_{-1}=1$ as there is an empty complex
contained in every complex. See \cite{cherngaussbonnet}. 
The formula appeared already in \cite{Levitt1992} but without seeing it as a Gauss-Bonnet result. 
Gauss-Bonnet makes sense for any simplicial complex $G$.  
If $v$ is a $0$-dimensional element in $G$ and $V(v)$ is the number of simplices containing $v$, then the same
curvature works. It can be formulated more generally for any element in strong ring. The curvatures just multiply
in the product: 

\begin{thm}[Gauss-Bonnet]
Given a ring element $G$.  The curvature function $K$ supported on the zero-dimensional part $V$ of $G$
satisfies $\sum_{v} K(v) = \chi(G)$. If $G = A \times B$, and $v=(a,b)$ is a $0$-dimensional point in $G$,
then $K_G(v) = K_A(a) K_B(b)$. 
\end{thm} 
\begin{proof}
The proof is the same. Lets take the product $A \times B$ of two simplicial complexes. 
We have $\sigma(A \times B) = \sum_{x,y} \omega(x) \omega(y)$, where the sum is over all
pairs $(x,y) \in A \times B$ (the set theoretical Cartesian product). The sum does not change, if we
distribute every value $\omega(x)$ equally to zero-dimensional subparts. This gives the curvature.
\end{proof} 

\paragraph{}
For Poincar\'e-Hopf \cite{poincarehopf}, we are given a locally injective function $f$ on $G$. 
Define the {\bf Poincar\'e-Hopf index} $i_f(v)$
at a $0$-dimensional simplex $v$ in $G$ as $1-\chi(S_f^-(x))$, where 
$S_f^-(v) = \{ x \in G \; | \; f(v)<f(x)$ and $v \subset x \}$ and the Euler characteristic is the usual 
sum of the $\omega(y)$, where $y$ runs over the set $S_f^-(v)$. This can now be generalized to products: 

\begin{thm}[Poincar\'e-Hopf]
Given a ring element $G$ and a locally injective function $f$ on $G$. The 
index function $i_f$ supported on the zero-dimensional part $V$ of $G$ satisfies
$\sum_v i_f(v) = \chi(G)$. If $G=A \times B$ and $v=(a,b)$ is a $0$-dimensional point in $G$
then $i_f(v) = i_f(a) i_f(b)$. 
\end{thm} 
\begin{proof}
Also here, the proof is the same. Instead of distributing the original curvature values $\omega(x) \omega(y)$
equally to all zero dimensional parts, it is only thrown to the zero dimensional simplex $(a,b)$ for
which the function is minimal on $(x,y)$. 
\end{proof}

\paragraph{}
Also index averaging generalizes. Given any probability measure $P$ on locally injective functions $f$,
one can look at the expectation $K_P(v) = {\rm E}[i_f(v)]$ which can now be interpreted as a curvature
as it does not depend on an indifidual function $f$ any more. There are various natural measures
which produce the Gauss-Bonnet curvature $K(x)$. One is to look at the product measure $[-1,1]$ 
indexed by $G$ \cite{indexexpectation}. 
An other is the set of all colorings, locally injective functions on $G$ \cite{colorcurvature}. 
Lets formulate it for colorings

\begin{thm}[Index averaging] 
Averaging $i_f(x)$ over all locally injective functions on $G$ 
with uniform measure gives curvature $K(x)$. 
\end{thm}

\paragraph{}
For Brouwer-Lefschetz \cite{brouwergraph}, we look at an endomorphisms $T$ 
of an element $G$ in the strong ring. The definition of the {\bf Brouwer index} is
the same as in the graph case: first of all, one can restrict to the attractor of $T$
and get an automorphism $T$. For a simples $x \in G$, define $i_T(x)={\rm sign}(T|x) \omega(x)$. 
Because $T$ induces a permutation on the simplex $x$, the signature of $T|x$ is defined. 
Also the definition of the {\bf Lefschetz number} $\chi_T(G)$ is the same. It is the super
trace on chomology 
$$ \chi_T(G) = \sum_{k=0} (-1)^k {\rm tr}(T|H^k(G)) \; . $$

\begin{thm}[Brouwer-Lefschetz]
$\sum_{x, T(x)=x} i_T(x) = \chi_T(G)$. 
\end{thm}
\begin{proof}
The fastest proof uses the heat flow $e^{ -tH(G)}$ for the Hodge Laplacian. 
The super trace ${\rm str}(H^k)$ is zero for $k>0$ by McKean-Singer \cite{knillmckeansinger}.
Define $l(t)={\rm str}(\exp(-tL) U_T)$, where 
$U_Tf=f(T)$ is the Koopman operator associated to $T$. The function $f(t)$ is constant. 
This heat flow argument proves Lefschetz because $l(0) = {\rm str}(U_T)$ is $\sum_{T(x)=x} i_T(x)$ 
and  $\lim_{t \to \infty} l(t)=\chi_T(G)$ by Hodge. 
\end{proof} 

\paragraph{}
There are more automorphisms $T$ in $A \times B$ than product automorphisms $T_1 \times T_2$. 
An example is if $A=B$ and $T( (x,y) ) = (y,x)$.  One could have the impression at first that
such an involution does not have a fixed point, but it does. Lets for example take $A=B=K_2$. 
The product $A \times B$ has $9$ elements and can be written as $(a+b+ab) (c+d+cd)$. 
The space is contractible so that only $H^0(G)$ 
has positive dimension and we are in the special case of the Brouwer fixed point case.
The Lefschetz number $\chi_T(G)$ is equal to $1$. There must be a fixed point. Indeed, it is
the two dimensional simplex $((a,b) \times (c,d))$ represented in the Stanley-Reisner picture 
as $abcd$. 

\section{Wu characteristic}

\paragraph{}
Euler characteristic $\chi(G)=\omega_1(G)$ is the first of a sequence $\omega_k$ of {\bf Wu characteristic}. 
The {\bf Wu characteristic} $\omega(G)=\omega_2(G)$ is defined for a simplicial complex $G$ as
$$  \sum_{x \sim y} \omega(x) \omega(y) , $$
where $\omega(x) = (-1)^{{\rm dim}(x)}$ and where the sum is taken over all intersecting simplices. 
The notation fits as $\omega(K_n) = (-1)^{n-1}$ which we have proven the fact that the Barycentric refinement
of the complete graph is a ball, a discrete manifold with sphere boundary of dimension $n-1$. A general 
formula for discrete manifolds with boundary then use the formula $\omega(G) = \chi(G) - \chi(\delta G)$.
Higher order versions $\omega_k(G)$ are defined similarly than $\omega(G)$. 
We just have to sum over all $k$-tuples of simultaneously intersecting simplices in the complex. 
While we have seen $\omega_k( (G \times H)_1 ) = \omega_k(G) \omega_k(H)$ and 
of course $\omega_k(G \oplus H) = \omega_k(G) + \omega_k(H)$, this insight was done for the Cartesian 
product $(G \times H)_1$ which was again a simplicial complex, the Whitney complex of a graph. 
The product property especially implies that the Barycentric refinement $G_1$ has the same Wu 
characteristics $\omega_k(G) = \omega_k(G_1)$. In other words, like Euler characteristic
$\chi=\omega_1$, also the Wu characteristic $\omega=\omega_2$ and higher Wu characteristics
$\omega_k(G)$ are {\bf combinatorial invariants}. 

\paragraph{}
The Wu characteristic can be extended to the strong ring. For simplicity, lets restrict to 
$\omega=\omega_2$. The notation $\omega(x) = (-1)^{{\rm dim}(x)}$ is extended to pairs of 
simplices as $\omega( (x,y) ) = \omega(x) \omega(y)$. So, $\omega$ is defined as a function
on the elements $(x,y)$ in the Cartesian product $G \times H$ of two simplicial complexes $G$
and $H$. We can not use the original definition of Wu characteristic for the product as
the product of two simplicial complexes is not a simplicial complex any more as the multiplicative
primes in the ring are the simplicial complexes. Lets write $(x,y) \sim (a,b)$ if both $x \cap a \neq \emptyset$
and $y \cap b \neq \emptyset$. Now define
$$ \omega(G \times H) = \sum_{(x,y) \sim (a,b)} \omega((x,y)) \omega((a,b)) \; . $$
As this is equal to $\sum_{(x,y) \sim (a,b)}$ $\omega(x) \omega(y) \omega(a) \omega(b)$ 
which is $\sum_{x \sim a} \sum_{y \sim b}$ $\omega(x) \omega(a) \omega(y) \omega(b)$  or
$(\sum_{x \sim a} \omega(x) \omega(a)) \sum_{y \sim b} \omega(y) \omega(b)$, which is
$\omega(G) \omega(H)$, the product property is evident. 
We can also define $\omega_k(-G) = -\omega_k(G)$ so that

\begin{propo}
All Wu characteristics $\omega_k$ are ring homomorphisms from the strong ring to $\ZZ$. 
\end{propo}

\paragraph{}
The just seen nice compatibility of Wu characteristic with the ring arithmetic structure
renders the Wu characteristic quite unique among multi-linear valuations \cite{valuation}. 
We have seen in that paper that for geometric graphs, there are analogue 
{\bf Dehn-Sommerville relations} which are other valuations which are zero,
answering a previously unresolved question of  \cite{Gruenbaum1970} from 1970. 
But this requires the simplicial complexes to be discrete manifolds in the sense that every unit 
sphere has to be a sphere. Dehn-Sommerville invariants are exciting that they defeat somehow the fate of
exploding in the Barycentric limit as they are zero from the beginning and
remain zero in the continuum limit. The local versions, the Dehn-Sommerville invariants of the unit spheres are local quantities which 
are {\bf zero curvature} conditions. One might wonder why Euler curvature is not defined for odd dimensional
manifolds for example. Indeed, Gauss-Bonnet-Chern is formulated only for even dimensional manifolds and
the definition of curvature involves a Pfaffian, which only makes sense in the even dimensional case. But 
what really happens is that there are curvatures also in the odd dimensional case, they are just zero
due to Dehn-Sommerville.  When writing \cite{cherngaussbonnet}, we were not aware of the Dehn-Sommerville
connection and had only conjectured that for odd dimensional geometric graphs the curvature is zero. It 
was proven in \cite{indexexpectation} using discrete integral geometry seeing curvature as an average of
Poincar\'e-Hopf indices. 

\paragraph{}
As a general rule, any result for Euler characteristic $\chi$ appears to generalize
to Wu characteristic. For Gauss-Bonnet, Poncar\'e-Hopf and index expectation linking the
two also the proofs go over. Start with the definition of Wu characteristic as a
Gauss-Bonnet type result where $\omega_k(x)$ is seen as a curvature on simplices. Then 
push that curvature down to the zero dimensional parts. Either equally, leading to a 
{\bf curvature}, or then directed along a gradient field of a function $f$, 
leading to {\bf Poincar\'e-Hopf indices}. Averaging over all functions, then essentially
averages over all possible ``distribution channels" $f$ leading for a nice measure on 
functions to a uniform distribution and so to curvature. The results and proofs generalize
to products.

\paragraph{}
Lets look at Gauss-Bonnet first for Wu characteristic: 

\begin{thm}[Gauss-Bonnet]
Given a ring element $G$. The curvature function $K_k$ supported on the zero-dimensional part $V$ of $G$
satisfies $\sum_{v} K_k(v) = \omega_k(G)$. If $G = A \times B$, and $v=(a,b)$ is a $0$-dimensional point in $G$,
then $K_k(v) = K_A(a) K_B(b)$.
\end{thm}

\paragraph{}
For formulating Poincar\'e-Hopf for Wu characteristic, 
we define for zero-dimensional entries $v=(a,b)$ the
stable sphere $S_f^-((a,b)) \{ (x,y) \in G \times H \; | \; f((a,b))<f((x,y))$ and 
$a \subset x, b \subset y \}$. This stable sphere is the join of the stable spheres.
The definition $i_{f,k}(v) = 1-\omega_k(S_f^-(v))$ leads now to
$i_{f,k}((a,b)) = i_{f,k}(a) i_{f,k}(b)$ and

\begin{thm}[Poincar\'e-Hopf]
Let $f$ be a Morse function, then $\omega_k(G) = \sum_v i_{f,k}(v)$,
where the sum is over all zero dimensional $v$ in $G$. 
\end{thm}

\paragraph{}
When looking at the index expectation results $K(x) = {\rm E}[i_f(x)]$, one 
could either directly prove the result or then note that if we look at a direct
product $G \times H$ and take probability measures $P$ and $Q$ on functions of $G$
and $H$, then the random variables $f \to i_f(x)$ on the 
two probability spaces $(\Omega(G),P)$ and $(\Omega(H),Q)$ are independent. 
This implies ${\rm E}[i_f(x) i_f(y)] = {\rm E}[i_f(x)] {\rm E}[i_f(y)]$ and
so index expectation in the product:

\begin{thm}[Index expectation]
If the probability measure is the uniform measure on all colorings, then
curvature $K_{G \times H,k}$ is the expectation of Poincar\'e-Hopf indices
$i_{G \times H,f,k}$.
\end{thm}

\paragraph{}
The theorems of Gauss-Bonnet, Poincar\'e-Hopf and index expectation
are not restricted to Wu characteristic. They hold for any {\bf multi-linear
valuation}. By the multi-linear version of the discrete Hadwiger theorem \cite{KlainRota},
a basis of the space of valuations is given. Quadratic valuations for 
example can be written as 
$$  X(G) = \sum_{x \sim y} X_{ij} V_{ij}(G) \; , $$
where $X$ is a symmetric matrix and where the {\bf $f$-matrix} $V_{ij}$ 
counts the number of pairs $x,y$ of $i$-dimensional simplices 
$x$ and $j$ dimensional simplices $y$ for which $x \cap y \neq \emptyset$.

\paragraph{}
To see how the $f$-vectors, the $f$-matrices and more generally the $f$-tensors
behave when we take products in the ring, its best to look at their 
{\bf generating functions}. Given a simplicial complex $G$ with 
$f$-vector $f(G) = (v_0(G),v_1(G), \dots)$, define the {\bf Euler polynomial}
$$ e_G(t) = \sum_{k=0}^{\infty} v_k(G) t^k $$
or the multi-variate polynomials like in the quadratic case
$$ V_G(t,s) = \sum_{k,l} V_{k,l}(G) t^k s^l \; , $$
which encodes the cardinalities $V_{k,l}(G)$ of intersecting $k$ and $l$ 
simplices in the complex $G$. The convolution of the $f$-vectors becomes
the product of Euler polynomials 
$$ e_{G \times H} = e_G e_H  \; . $$

\paragraph{}
If we define the $f$-vector of $-G$ as $-f(G)$ and so the Euler polynomial
of $-G$ as $-p_G$, we can therefore say that

\begin{propo}
The Euler polynomial extends to a ring homomorphism from
the strong ring to the polynomial ring $\ZZ[t]$. 
\end{propo}

\paragraph{}
This also generalizes to the multivariate versions. To the $f$-tensors
counting $k$-tuple intersections in $G$, we can associate polynomials in 
$\ZZ[t_1, \dots, t_k]$ which encode the $f$-tensor and then have

\begin{propo}
For every $k$, the multivariate $f$-polynomial construction
extends to ring homomorphisms from the strong ring to $\ZZ[t_1, \dots, t_k]$. 
\end{propo}

\paragraph{}
We see that like in probability theory, where moment generating functions or
characteristic functions are convenient as they ``diagonalize" the combinatorial
structure of random variables on product spaces (independent spaces), the 
use of $f$-polynomials helps to deal with the combinatorics of the $f$-tensors
in the strong ring. 

\paragraph{}
In order to formulate results which involve cohomology like the Lefschetz 
fixed point formula, which reduces if $T$ is the identity to the Euler-Poincar\'e
formula relating combinatorial and cohomological Euler characteristic, 
one has to define a cohomology. Because the name intersection cohomology is taken,
we called it {\bf interaction cohomology}. It turns out that this cohomology is
finer than simplicial cohomology. Its Betti numbers are combinatorial invariants
which allow to distinguish spaces which simplicial cohomology can not, the prototype
example being the cylinder and Moebius strip
\cite{CaseStudy2016}. The structure
and proofs of the theorems however remain. The heat deformation proof of Lefschetz 
is so simple that it extends to the ring and also from Euler characteristic to 
Wu characteristic.

\paragraph{}
The definition of interaction cohomology involves
explicit matrices $d$ as exterior derivatives.  
The {\bf quadratic interaction cohomology} for example is defined through 
the {\bf exterior derivative} 
$dF(x,y) = F(\delta x,y) + (-1)^{{\rm dim}(x)} F(x,\delta y)$ on functions $F$
on ordered pairs $(x,y)$ of intersecting simplices in $G$. This generalizes
the exterior derivative $dF(x)=F(\delta x)$ of simplicial cohomology. 
This definition resembles the de-Rham type definition of exterior derivative
for the product of complexes, but there is a difference: in the interaction
cohomology we only look at pairs of simplices which interact (intersect). 
In some sense, it restricts to the simplices in the ``diagonal" of the product. 

\paragraph{}
It is obvious how to extend the definition of interaction cohomology to the product
of simplicial complexes and so to the strong ring. It is still important to 
point out that at any stage we deal with finite dimensional matrices. 
The {\bf quadratic interaction exterior derivative } $d$ is defined as 
$dF(x,y) = d_1 F + (-1)^{\rm dim(x)} d_2 F$, where $d_1$ is the {\bf partial 
exterior derivative} with respect to the first variable (represented by simplices
in $G$) and $d_2$ the partial exterior derivative with respect to the second variable
(represented by simplices in $H$). These partial derivatives are given by the 
intersection exterior derivatives defined above. 

\paragraph{}
Lets restrict for simplicity to quadratic interaction cohomology in which 
the Wu characteristic $\omega=\omega_2$ plays the role of the Euler characteristic 
$\chi = \omega_1$. Let $G$ first be a simplicial complex. 
If $b_p(G)$ are the Betti numbers of these interaction cohomology groups of $G$,
then the {\bf Euler-poincar\'e} formula $\omega(G)=\sum_p (-1)^p b_p(G)$ holds. 
More generally, the {\bf Lefschetz formula} 
$\chi_T(G)=\sum_{(x,y)=(T(x),T(y))} i_T(x,y)$ generalizes, where $\chi_T(G)$ is the {\bf Lefschetz number}, 
the super trace of the Koopman operator $U_T$ on cohomology and where 
$i_T(x,y) = (-1)^{{\rm dim}(x) + {\rm dim}(y)} {\rm sign}(T|x) {\rm sign}(T|y)$ is
the {\bf Brouwer index}. The heat proof generalizes.

\paragraph{}
The interaction cohomology groups are defined similarly for the product of simplicial
complexes and so for general elements in the strong ring. 
The K\"unneth formula holds too. We can define the Betti numbers $b_{k}(-G)$ as $-b_k(G)$.
The {\bf interaction Poincar\'e polynomial} $p_G(x)=\sum_{k=0} {\rm dim}(H^k(G)) x^k$ again
satisfies $p_{G \times H}(x) = p_G(x) p_H(x)$ so that 

\begin{thm}
For any $k$, the interaction cohomology polynomial extends to a
ring homomorphism from the strong ring to $\ZZ[t]$. 
\end{thm}

\paragraph{}
While we hope to be able to explore this more elsewhere we note for now just that 
all these higher order interaction cohomologies associated to the Wu characteristic 
generalize from simplicial complexes to the strong ring. Being able to work on the
ring is practical as the interaction exterior derivatives are large matrices. So
far we had worked with the Barycentric refinements of the products, where the matrices
are bulky if we work with a full triangulation of the space. Similarly as de-Rham cohomology
significantly cuts the complexity of cohomology computations, this is also here the
case in the discrete. For a triangulation of the cylinder $G$, the full Hodge Laplacian $(d+d^*)^2$
of quadratic interaction cohomology is a $416 \times 416$ matrix as there were
$416  = \sum_{i,j} V_{ij}(G)$ pairs of intersecting simplices. When working in the ring, we
can compute a basis of the cohomology group from the basis of the circles and be done much faster.

\paragraph{}
For the M\"obius strip $G$ however, where the Hodge Laplacian of a triangulation leads to $364 \times 364$
matrices, we can not make that reduction as $G$ is not a product space. The smallest
ring element which represents $G$ is a simplicial complex. Unlike the cylinder
which is a ``composite particle", $G$ is an ``elementary particle".
In order to deal with concrete spaces, one would have to use patches of Euclidean
pieces and compute the cohomology using Mayer-Vietoris. 

\paragraph{}
What could be more efficient is to
patch the space with contractible sets and look at the cohomology of a nerve graph which is 
just \v{C}ech cohomology. But as in the M\"obius case, we already worked 
with the smallest nerve at hand, this does not help in the computation in that case. It would
only reduce the complexity if we had started with a fine mesh representing the geometric object
$G$. We still have to explore what happens to the $k-$harmonic functions in the
kernel of the Hodge blocks $H_k$ and the spectra of $H_k$ in the interaction cohomology case, 
if we cut a product space and glue it with reverse orientation as in $G$. 

\paragraph{}
There is other geometry which can be pushed over from 
complexes to graphs. See \cite{knillcalculus,KnillILAS,KnillBaltimore} for snapshots
for results formulated for Whitney complexes of graphs 
from 2012, 2013 and 2014. The {\bf Jordan-Brouwer theorem} for
example, formulated in \cite{KnillJordan} for Whitney
complexes of graphs generalizes to simplicial complexes 
and more generally to products as we anyway refer to the Barycentric
refinement there which is always a Whitney complex.

\paragraph{}
Also promising are subjects close to calculus like the discrete {\bf Sard theorem}
\cite{KnillSard} which allow to define new spaces in given spaces by looking
at the zero locus of functions. Also this result was formulated in graph theory
but holds for any simplicial complex. 
The {\bf zero locus} $\{ f=c \}$ of a function $f$ on a ring element $G$ can be defined
as the complex obtained from the cells where $f$ changes sign.
We nave noticed that if $G$ is a discrete $d$-manifold in the sense that every
unit sphere in the Barycentric refinement is a $(d-1)$-sphere, then 
for any locally injective function $f$ and any $c$ different from the range of $f$
the zero locus $f=c$ is a discrete manifold again.

\section{Two siblings: the Dirac and Connection operator}

\paragraph{}
To every $G$ in the strong ring belong two graphs $G_1$ and $G'$, the Barycentric refinement
and the connection graph. The addition of ring elements produces disjoint unions of graphs. The multiplication
naturally leads to products of the Barycentric refinements $A_1 \times B_1=  (A \times B)_1$
as well as connection graphs $A' \osquare B' = (A \times B)'$, where $\osquare$ is the 
strong product.

\paragraph{}
An illustrative example is given in \cite{CountingAndCohomology}. Both the {\bf prime graph} $G_1$
as well as the {\bf prime connection graph} have as the vertex set the set of square free integers
in $\{2,3,\dots,n\}$. In $G_1$, two numbers are connected if one is a factor of the other.
It is part of the Barycentric refinement of spectrum of the integers. 
In the {\bf prime connection graph} $G'$ two integers are connected if they have a common factor 
larger than $1$. It has first appeared in \cite{Experiments}. This picture sees square free integers
as simplices in a simplicial complex. The number theoretical M\"obius function $\mu(k)$ has the 
property that $-\mu(k)$ is the Poincar\'e-Hopf index of the counting function $f(x)=x$. 
The Poincare-Hopf theorem is then $\chi(G)=1-M(n)$, where $M(n)$ is the Mertens function. 
As the Euler characteristic $\sum_x \omega(x)$ can also be expressed through Betti numbers, there is a relation
between the Mertens function and the kernels of Hodge operators $D^2$. On the other hand, 
the Fermi characteristic $\prod_x \omega(x)$ is equal to the determinant of the connection Laplacian.
This was just an example. Lets look at it in general. 

\paragraph{}
To every $G \in R$ belong two operators $D$ and $L$, the Dirac operator and connection operator.
They are both symmetric matrices acting
on the same Hilbert space. While $D(-G)=-D(G)$ does not change anything in the spectrum of $G$, 
a sign change of $G$ changes the spectrum as $L(-G)=-L(G)$ is a different operator due to lack 
of symmetry between positive and negative spectrum. There are higher order Dirac operators $D$
which belong to the exterior derivative $d$ which is used in interaction cohomology. 
The operator $D$ which belongs to the Wu characteristic for example does not seem to have any
obvious algebraic relation to the Dirac operator belonging to Euler characteristic. Indeed, as
we have seen, the nullity of the Wu Dirac operator is not homotopy invariant in general. 
On the other hand, the interaction Laplacian $L$ belonging to pairs of interacting simplices
is nothing else than the tensor product $L \otimes L$, which is the interaction Laplacian of 
$G \times G$.

\paragraph{}
There are various indications that the Dirac operator $D(G)$ is of {\bf additive nature} while
the connection operator $L(G)$ is of {\bf multiplicative nature}. One indication is that $L(G)$ is 
always invertible and that the product of ring elements produces a tensor product of 
connection operators.  Let ${\rm str}$ denote the super trace
of a matrix $A$ defined as ${\rm str}(A) = \sum_x \omega(x) A_{xx}$, where 
$\omega(x) = (-1)^{{\rm dim}(x)}$. The discrete version of the {\bf Mc-Kean Singer formula} \cite{McKeanSinger}
can be formulated in a way which makes the additive and multiplicative nature of $D$ and $L$ clear:

\begin{thm}[Mc Kean Singer]
${\rm str}(e^{-D})=\chi(G)$ and ${\rm str}(L^{-1}) = \chi(G)$. 
\end{thm}
\begin{proof}
The left identity follows from the fact that ${\rm str}(D^{2k})=0$ for every even $k$
different from $0$, that for $k=0$, we have the definition of $\chi(G)$ and that for odd
$2k+1$ the diagonal entries of $D^k$ are zero. The right equation is a Gauss-Bonnet formula
as the diagonal entries $L^{-1}_{ii}$ are the Euler characteristics $\chi(S(x))$ of unit 
spheres so that $\omega(x) \chi(S(x)) = 1-\chi(S^-_f(x))$ is a Poincar\'e-Hof index
for the Morse function $f(x)=-{\rm dim}(x)$. 
\end{proof} 

These identities generalize to a general ring element $G$ in the strong ring. 

\paragraph{}
A second indication for the multiplicative behavior of quantities related to the connection
graph is the Poincar\'e-Hopf formula. The multiplicative analogue of the Euler characteristic
$\chi(G) = \sum_x \omega(x)$ is the {\bf Fermi characteristic} $\phi(G) = \prod_x \omega(x)$. 
Lets call a function $f$ on a simplicial complex a {\bf Morse function} if $S^-_f(x) = \{ y \in S(x) \; | \; 
f(y)<f(x) \}$ is a discrete sphere. Since spheres have Euler characteristic  in $\{0,2\}$ the
Euler-Poincar\'e index $i_f(x) = 1-\chi(S^-_f(x))$ is in $\{ -1,1\}$. 

\begin{thm}[Poincar\'e-Hopf]
Let $f$ be a Morse function, then $\chi(G) = \sum_x i_f(x)$ and
$\phi(G) = \prod_x i_f(x)$. 
\end{thm}

Since for any simplicial complex $G$, we can find a Morse function, the multiplicative part can be used
to show that $\psi(G) = {\rm det}(L(G))$ is equal to $\phi(G)$. This is the {\bf unimodularity theorem}.
As we have seen, this theorem generalize to the strong ring. 

\paragraph{}
The most striking additive-multiplicative comparison of $D$ and $L$ comes through spectra. 
The spectra of the Dirac operator $D$ behave additively, while the spectra of connection 
operators multiply. The additive behavior of Hodge Laplacians happens also classically when looking at
Cartesian products of manifolds: 

\begin{thm}[Spectral Pythagoras]
If $\lambda_i$ are eigenvalues of $D(G_i)$, there is an eigenvalue $\lambda_{ij}$ of $D(G_i \times G_j)$
such that $\lambda_i^2 + \lambda_j^2 = \lambda_{ij}^2$. 
\end{thm}

\paragraph{}
We should compare that with the multiplicative behavior of the connection operators:

\begin{thm}[Spectral multiplicity]
If $\lambda_i$ are eigenvalues of $L(G_i)$, then there is an eigenvalue $\lambda_{ij}$ of $L(G_i \times G_j)$
such that $\lambda_i \lambda_j = \lambda_{ij}$. 
\end{thm}

\paragraph{}
In some sense, the operator $L$ describes energies which are multiplicative and so can become much larger
than the energies of the Hodge operator $H$. Whether this has any significance in physics is not clear.
So far this is pure geometry and spectral theory. The operator $L$ does not have this block structure 
like $D$. When looking at Schr\"odinger evolutions $e^{iDt}$ or $e^{i Lt}$, we expect different 
behavior. Evolving with $L$ mixes parts of the Hilbert space which are separated in the Hodge case. 

\section{Dimension}

\paragraph{}
The {\bf maximal dimension} ${\rm dim}_{{\rm max}}(G)$ of a simplicial complex $G$ is defined as $|x|-1$, where $|x|$ 
is the {\bf cardinality} of a simplex $x$. The {\bf inductive dimension} of $G$ is defined as
the inductive dimension of its Barycentric refinement graph $G_1$ (or rather the Whitney complex
of that graph) \cite{elemente11}. 
The definition of inductive dimension for graphs is recursive: 
$$   {\rm dim}(x) = 1+|S(x)|^{-1} \sum_{y \in S(x)} {\rm dim}(y) $$ 
starting with the assumption that the dimension of the empty graph is $-1$. 
The unit sphere graph $S(x)$ is the subgraph of $G$ generated by all vertices directly connected to $x$. 

\begin{propo}
Both the maximal and inductive dimension can be extended 
so that additivity holds in full generality for non-zero elements in the strong ring. 
\end{propo}

The zero graph $\emptyset$ has to be excluded as when we multiply a ring element with the
zero element, we get the zero element and ${\rm dim}(0)=-1$ would imply 
${\rm dim}(G \times 0)={\rm dim}(0) = -1$
for any $G$. The property $0 \times G = 0$ has to hold if we want the ring axioms to hold. 
We will be able to extend the clique number to the dual ring of the strong connection ring 
but not the dimension. 

\paragraph{}
Since the inductive dimension satisfies
$$   {\rm dim}( (G \times H)_1) = {\rm dim}(G_1) + {\rm dim}(H_1) \; , $$
(see \cite{KnillKuenneth}), the definition
$$ {\rm dim}(G \times H) = {\rm dim}(G) + {\rm dim}(H) $$
is natural. We have also shown ${\rm dim}(G)_1 \geq {\rm dim}(G)$ so that
for nonzero elements $G$ and $H$: 
${\rm dim}(G \times H)_1 \geq {\rm dim}(G) + {\rm dim}(H)$, an inequality which looks
like the one for Hausdorff dimension in the continuum. Again, also this does not 
hold if one of the factors is the zero element. 

\paragraph{}
In comparison, we have seen that the Zykov sum has the property that the 
{\rm clique number} ${\rm dim}_{{\rm max}}(G)+1$ is additive. 
and that the large Zykov product, which is the dual of the strong product has 
the property that the clique number is multiplicative. 
If we define the clique number of $-G$ as minus the clique number of $G$, then we have

\begin{propo}
The clique number from the dual $R^*$ of the strong ring $R$ to the integers $\ZZ$
is a ring homomorphism. 
\end{propo}
\begin{proof}
By taking complement, the identity $P_n \star P_m = P_n \osquare P_m = P_{nm}$ gives
$K_n \osquare K_m = K_m \osquare K_n = K_{nm}$.
We have justified before how to extend the clique number to $c(-G)=-c(G)$.
\end{proof}

\paragraph{}
The clique number is a ring homomorphism, making the clique number 
to become negative for negative elements a natural choice. 
For dimension, this is not good:
the dimension of the empty graph $0$ is $-1$ and since $0=-0$ this is not good.
If we would take a one dimensional graph $G$, we would have to define 
${\rm dim}(-G)=-1$ but that does not go well with the fact that the zero graph
has dimension $-1$. We want the inductive dimension to be the average
of the dimensions of the unit spheres plus $1$ which makes the choice
${\rm dim}(-G)= {\rm dim}(G)$ natural. 

\paragraph{}
The {\bf maximal dimension} ${\rm dim}_{\rm max}(G)$ of a complex is the dimension
of the maximal simplex in $G$. The {\bf clique number} $c(G) = {\rm dim}_{\rm max}(G)+1$ is the
largest cardinality which appears for sets in $G$. For graphs, it is the largest
$n$ for which $K_n$ is a subgraph of $G$. The empty complex or empty graph has
clique number $0$. Extend the clique number to the entire ring by defining $c(-G)=-c(G)$.
The strong ring multiplication satisfies $c(G \osquare H) = c(G) c(H)$ and
$c(G \oplus H) = {\rm max}(c(G),c(H))$.

\section{Dynamical systems}

\paragraph{}
There is an isospectral {\bf Lax deformation} $D'=[B,D]$ with $B=(d-d^*)$
of the Dirac operator $D=(d+d^*)$ of a simplicial complex. This works now also for any 
ring elements. The deformed Dirac operator is then of the form $d+d^* + b$
meaning that $D$ has additionally to the geometric exterior derivative part also a diagonal 
part which is geometrically not visible. What happens is that the off diagonal part $d(t)$ 
decreases, leading to an expansion of space if we use the exterior derivative to measure distances.
The deformation of exterior derivatives is also defined for Riemannian manifolds but the 
deformed operators are pseudo differential operators. 

\paragraph{}
A complex generalization of the system is obtained by defining
$B=d-d^*+i \beta b$, where $\beta$ is a parameter. 
This is similar to \cite{Toda} who modified the Toda flow by adding $i \beta$ to 
$B$. The case above was the situation $\beta=0$, where the flow has a situation
for scattering theory. The case $\beta=1$ leads asymptotically to a linear 
wave equation. The exterior derivative has become complex however. All this
dynamics is invisible classically for the Hodge operator $H$ as the Hodge operator
$D^2$ does not change under the evolution.

\paragraph{}
As any Lax pair does, the eigenvalues are integrals of motion. The equation $U'(t) = B(t) U(t)$ 
produces a unitary curve $U(t)$ which has the property that if $D_1=V D V^T$ corresponds to an other choice of 
simplex orientation then $D_1(t) = V D(t) V^T$ carries on.
The evolution does not depend on the gauge (the choice of signs used to make the simplicial 
complexes signed) but the energy for $L(t) = U(t) L U(t)^*$ does.

\begin{thm}
For every $G$ in the strong ring, we have nonlinear integrable Hamiltonian systems
based on deformation of the Dirac operator. 
\end{thm}

This is very concrete as for every ring element we can write down a concrete matrix 
$D(G)$ and evolve the differential equation. 
As for simplicial complexes, also when we deform a $D$ from the strong ring, 
the Hodge Laplacian $H(G) = D(G)^2$ does not move under the deformation. 
Most classical physics therefore is unaffected by the expansion. 
For more details see \cite{KnillBarycentric,KnillBarycentric2}. 

\paragraph{}
Any type of Laplacian produces Schr\"odinger type evolutions. Examples of Laplacians are
the {\bf Kirchhoff matrix} $H_0$ of the graph which has $G$ as the Whitney complex or 
the Dirac matrix $D=(d+d^*)$ or the Hodge Laplacian $D^2$ or the connection graph $L$ of $G_n$. 
One can also look at non-linear evolutions. An example is the nonlinear 
Schr\"odinger equation obtained from a functional like $F(u) = \langle u,g u \rangle - V(|u|^2)$ 
on the Hilbert space. The Hamiltonian system is then $u' = \partial_{\overline{u}} F'(u)$. 
This {\bf Helmholtz evolution} has both the energy $F(u)$ as well as $|u|^2$ as integrals of
motion. We can therefore restrict to {\bf states}, vectors of 
length $1$. A natural choice compatible with the product structure is the 
{\bf Shannon entropy} $V(|u|^2) = \beta S(u) = - \beta \sum_x p(x) \log(p(x))$ where $p(x)=|u(x)|^2$. 
Summands for which $p(x)=0$ are assumed to be zero as $\lim_{p \to 0}  p \log(p)=0$. 

\paragraph{}
We like the Helmhotz system because the Shannon entropy is essentially unique in the property that it 
is additive with respect to products. As $\langle 1,g 1 \rangle = \chi(G)$ is the Euler characteristic,
which is compatible with the arithmetic, the {\bf Helmholtz free energy} $F(u)$ leads to a
natural Hamiltonian system. The {\bf inverse temperature} $\beta$ can be used as a perturbation parameter.
If $\beta=0$, we have a Schr\"odinger evolution. If $\beta>0$, it becomes a nonlinear Schr\"odinger
evolution. So far, this system is pretty much unexplored. The choice of both the energy and entropy
part was done due to arithmetic compatibility: Euler characteristic is known to be the unique valuation
up to scale which is compatible with the product and entropy is known to be a unique quantity 
(again up to scale) compatible with the product by a theorem of Shannon.  See \cite{Helmholtz}
were we look at bifurcations of the minima under temperature changes. There are {\bf catastrophes}
already for the simplest simplicial complexes. 

\section{Barycentral central limit}

\paragraph{}
We can use a Barycentric central limit theorem to
show that in the {\bf van Hove limit} of the strong nearest neighbor
lattice $Z \times \cdots \times Z$, the spectral measure of the connection Laplacian $L$
has a {\bf mass gap}: not only $L$, but also the Green function, the inverse $g=L^{-1}$
has a bounded almost periodic infinite volume limit. The potential theory of the
Laplacian remains bounded and nonlinear time-dependent partial difference equations like
$Lu+ cV(u)=W$ which are Euler equations of {\bf Frenkel-Kontorova} type variational problems
have unique solutions $u$ for small $c$.

\paragraph{}
A {\bf Fock space} analogy is to see an additive prime in the ring (a connected space) as a {\bf particle state}.
The sum of spaces is then a collection of independent particles and the product is an entangled
{\bf multi-particle system}. Every independent particle has a unique decomposition into
{\bf multiplicative primes} which are the ``elementary particles". But the union of two particles can decay
into different type of elementary particles. For an entangled multi-particle system $G \times H$,
the spectrum of that configuration consists of all values $\lambda_i \mu_j$, where $\lambda_i$
are the eigenvalues of $L(G)$ and $\mu_j$ are the eigenvalues of $L(H)$.

\paragraph{}
While $G \osquare H =H \osquare G$, we don't have $L(G) \otimes L(H) 
= L(H) \otimes L(G)$ as the tensor product of matrices is not commutative.
But these two products are unitary equivalent.
We can still define an {\bf anti-commutator} $L(H) \otimes L(G) - L(G) \otimes L(H)$.
We observed experimentally that the kernel of this anti-commutator is non-trivial only if both
complexes have an odd number of vertices. While $L(G)$ can have non-simple spectrum,
we have so far only seen simple spectrum for operators $L(G)\times L(H)-L(H) \otimes L(G)$.

\paragraph{}
When we take successive Barycentric refinements $G_n$ of complex $G$ we see 
a universal feature \cite{KnillBarycentric,KnillBarycentric2}.  This generalizes 
readily to the ring.

\begin{thm}[Barycentric central limit]
For any $G$ in the ring and for each of the operators $A=L$ or $A=H$ or $A=H_k$, 
the density of states of $A(G_n)$ converges weakly to a measure which only 
depends on the maximal dimension of $G$. 
\end{thm}

\paragraph{}
The reason why this is true is that the $(k+1)$'th 
Barycentric refinement of a maximal $d$-simplex consists of $(d-1)!$ smaller disks glued along 
$d$-dimensional parts which have a cardinality growing slower. 
The gluing process changes a negligible amount of matrix entries. This 
can be estimated using a result of Lidskii-Last \cite{SimonTrace}. 
Examples of Laplacians for which the result holds are the 
Kirchhoff matrix of the graph which has $G$ as the Whitney complex or 
the Dirac matrix $D=(d+d^*)$ or the Hodge Laplacian $D^2$ 
or the connection graph $L$ of $G_n$.

\paragraph{}
In the one-dimensional case with Kirchhoff Laplacian, 
the limiting measure is derivative of the inverse of $F(x)=4 \sin^2(\pi x/2)$. 

\paragraph{}
For $G=C_n$ we look at the eigenvalues of the $v$-torus $C_n^{\rm nu} = C_n \osquare C_n \cdots \osquare C_n$. 
The spectrum defines a discrete measure $dk_n$ which has a weak limit $dk$.

\begin{thm}[Mass gap] 
The weak limit $dk_n$ exists, is absolutely continuous and has support away 
from $0$. The limiting operator is almost periodic, bounded and invertible. 
\end{thm}

\begin{proof}
For $\nu=1$, the density of states has a support which contains two intervals. 
The interval $[-1/5,1/5]$ is excluded as we can give a bound on $g=L^{-1}$. 
For general $\nu$, the gap size estimate follows from the fact that under
products, the eienvalues of $L$ multiply. 
\end{proof}

\section{Stability in the infinite volume limit}

\paragraph{}
To the lattice $Z^\nu$ belongs  the Hodge Laplacian 
$H_0 u(n) =$ $\sum_{i=1}^{\nu} u(n+e_i)-u(n)$. In the one-dimensional case, where
$H_0 u(n)=u(n+1)-2u(n)-u(n-1)$ corresponds after a Fourier transform to the multiplication
by $2 \cos(x)-2$ we see that the spectrum of $H_0$ is $[0,4]$. The non-invertibility leads
to ``small divisor problems". Also, the trivial linear solutions $u(n) = \theta + \alpha n$ 
to $Lu=0$ are minima. Its this minimality which allows the problem to be continued to 
nonlinear situations like $u(n+1)-2u(n)-u(n-1) = c \sin(u(n))$ for Diophantine $\alpha$.
For $\alpha=p/q$, Birkhoff periodic points and for general
$\alpha$, minimizers in the form of Aubry-Mather sets survive. For the connection Laplacian $L$
$0$ is in a gap of the spectrum. 

\paragraph{}
Unlike for the Hodge Laplacians in $Z^d$ which naturally are expressed on 
the Pontryagin dual of $T^d$, for which the Laplacian has spectrum containing $0$, we 
deal with the dual of is $\DD_2^{\nu}$, where $\DD_2$ is the {\bf dyadic group}. The limiting operator 
of $L$ is almost periodic operator and has a bounded inverse. Whereas in the Hodge case without mass gap, a 
{\bf strong implicit function theorem} is required to continue solutions of nonlinear Frenkel-Kontorova 
type Hamiltonian systems $Lu+ \epsilon V(u)=0$, the connection Laplacian is an invertible kinetic part 
and perturbation theory requires only the weak implicit function theorem. Solutions of the Poisson equation $L u = \rho$ 
for example can be continued to nonlinear theories $(L+V) v = \rho$. Since Poisson equations have unique solutions,
also the discrete Dirichlet has unique solutions. The classical Standard map model $L_0 u  + c \sin(u) = 0$
for example which is the Chirikov map in in 1D, requires KAM theory for solutions $u$ to exist. Weak
solutions also continue to exist by {\bf Aubry-Mather theory} \cite{MoserVariations}. 

\paragraph{}
If the Hodge Laplacian is replaced by the connection Laplacian,
almost periodic solutions to driven systems continue to exist for small $c$ similarly as
the {\bf Aubry anti-integrable limit} 
does classically for large $c$. But these continuations are in general
not interesting. Any Hamiltonian system with Hamiltonian $L+V$, where the kinetic energy is the connection Laplacian
remains simple. If we look at $L u = W$ where $W$ is obtained from an continuous function on the dyadic integers,
then the solution $u$ is given by a continuous function on the dyadic integers. This could be of some
interest but there is no interpretation of this solution as an orbit of a Hamiltonian system. 

\begin{coro}
A nonlinear discrete difference equation $Lu+\epsilon V(u)  = g$ has a unique almost periodic
solution if $g(n) = g(T^nx)$ is obtained from $Z^{\nu}$ action on $\DD_2^{\nu}$.
\end{coro}
\begin{proof}
For $\epsilon=0$, we have the Poisson equation $Lu=g$ which has the solution $u=L^{-1} g$,
where $u(n)=h(T^nx)$ is an almost periodic process with $h \in C(\DD_2^{\nu})$.      
Since $L$ is invertible, the standard implicit function theorem allows a continuation for small
$\epsilon$.
\end{proof}

\paragraph{}
It could be useful to continue a Hamiltonian to the infinite limit.
An example is the Helmholtz Hamiltonian like $H(\psi) = (\psi, g \psi)$, where $g=L^{-1}$.
Since the Hessian $g$ is invertible, we can continue a minimal solution to $H(\psi) + \beta V(\psi)$
for small $\beta$ if $V$ is smooth. 
We suspect that we can continue almost periodic solutions to the above defined Helmholtz system 
$V(\psi) = \beta S(|\psi|^2)$ with entropy $S$ for small $\beta$ but there is a technical difficulty
as $p \to p \log|p|$ is not smooth at $p=0$, only continuous. One could apply the weak implicit
function theorem to continuous almost periodic functions on $X=D_2^{\nu}$ if
$f \to S(|f|)$ was Fr\'echet differentiable on the Banach space $X$ but we have not proved that. 
It might require to smooth out $S_{\epsilon}$ first then show that the solution survives in the limit
$\epsilon \to 0$. 

\section{A pseudo Riemannian case}

\paragraph{}
When implementing a speudo Riemannian metric signature like $(+,+,+,-)$ on the lattice $\ZZ^4$, this 
changes the sign of the corresponding Hodge dual $d^*$. The Dirac operator of that coordinate axes 
is now $d_i - d_i^*$ and the corresponding Hodge operator is $D^* D = -H$ has just changed sign. 
The connection Laplacian $L$ is not affected by the change of Riemannian metric. 
Having a Pseudo Riemannian metric on $\ZZ^4$, we can look
at kernel elements $Hu=0$ as solutions to a discrete wave equation. Global solutions on a 
compact space like the product of circular graphs $C_n^4=C_n \times C_n \times C_n \times C_n$ 
are not that interesting. On an infinite lattice, we could prescribe
solutions on the space hypersurface $t=0$ and then continue it. Technically this leads to
a {\bf coupled map lattice}. \\

\paragraph{}
If $H$ is the operator of the $4$-torus $C_n^4$ with Lorentzian metric signature, then the 
eigenvalues are all of the form 
$\lambda_1 + \lambda_2+\lambda_3 - \lambda_4$, where $\lambda_i$ are eigenvalues of $C_n$. 
This just shifts the eigenvalues. This completely answers also the question what the
limiting density of states.

\begin{propo}
The spectrum of the Hodge operator of the $\nu$-torus $\TT_n^{\nu}$ with 
Lorentz signature $(m,k)$ agrees with the spectrum of the Hodge operator for the
signature $(\nu,0)$ shifted by $-4k$. 
\end{propo}
\begin{proof}
For $\nu=1$, replacing $H$ with $-H$ has the effect that that the spectrum
changes sign. But this is $\sigma(H)-4$. When taking the product, we get 
a convolution of spectra which commutes with the translation.
As for $C_n$, the spectrum satisfies $\sigma(-H)=\sigma(H)-4$, the switch
to a Lorentz metric goes over to the higher products. 
\end{proof}

\paragraph{}
As the Lorentz space appears in physics, the geometry of the
$\ZZ^4$ lattice with Lorentz signature metric $(+,+,+,-)$ is of interest.
The Barycentric limit leads to more symmetry in this discrete lattice case.
The limiting operators $D$ and $L$ are almost periodic on $G_2^{4}$, 
the compact group of dyadic integers. Besides of the group translations, there
are also {\bf scaling symmetries}. By allowing both scaling transformations
and group translations, we can implement symmetries which approximate Euclidean
symmetries in the continuum. The obstacle of poor symmetry properties 
in a discrete lattice appears to disappear in the Barycentric limit. 

\section{Illustrations}

\paragraph{}
Here is an illustration of an element $G=C_4-2 K_3 + (L_2 \times L_3)$ 
in the strong ring. In the Stanley-Reisner picture, we can write the ring 
element as
\begin{eqnarray*}
   f_G &=& a+ab+b+bc+c+cd+d+ad  - 2(x+y+z+xy+xz+yz)  \\
       &+& (u+uv+v) (p+pq+q+qr+r) \; . 
\end{eqnarray*}
The Euler characteristic is $\chi(G) = -f_G(-1,-1, \dots) = -1$. 

\paragraph{}
The ring element $G=C_4-2 K_3 + (L_2 \times L_3)$ in the strong ring
is a sum of three parts, where the first is a Whitney complex of a graph,
the second is $-2$ times a Whitney complex and the third is a product of
two Whitney complexes $L_2$ and $L_3$. 

\paragraph{}
In Figure~({\ref{example}), we drew the weak Cartesian product to visualize $L_2 \times L_3$. 
In reality, $L_2 \times L_3$ is not a simplicial complex. There are 6 two-dimensional 
{\bf square cells} present, one for each of the $6$ holes present in the weak product. 
A convenient way to fill the hole is to look at the Barycentric refinement 
$(L_2 \times L_3)_1$ as done in \cite{KnillKuenneth}. This is then a {\bf triangulation} 
of $L_2 \times L_3$. 

\paragraph{}
In Figure~({\ref{exampleconnection})
we see the connection graph $G'$ to the ring element $G$.
It is a disjoint union
of graphs, where the connection graphs to $K_3$ are counted negatively.
The connection Laplacian is
$L(G) = L(C_4) \oplus (-L(K_3)) \oplus (-L(K_3)) \oplus [ L(L_2) \otimes L(L_3) ]$.
Figure~({\ref{examplebarycentric}) shows the Barycentric refinement graph 
$G_1$. 

\begin{figure}[!htpb]
\scalebox{0.94}{\includegraphics{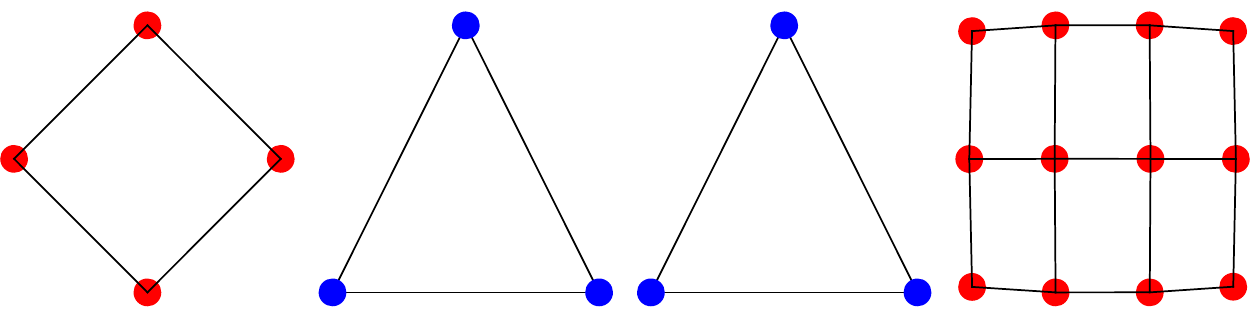}}
\label{example}
\caption{
$G=C_4-2 K_3 + (L_2 \times L_3)$ in the strong ring.
}
\end{figure}

\begin{figure}[!htpb]
\scalebox{0.94}{\includegraphics{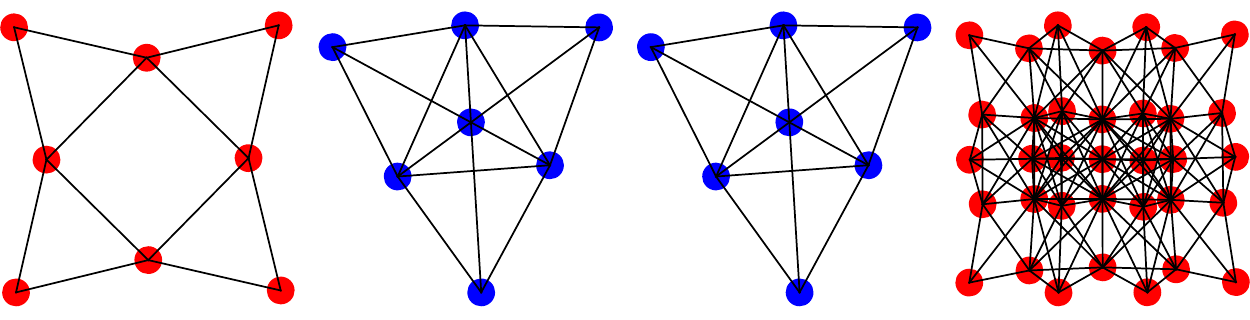}}
\caption{
The connection graph of $G$.
}
\label{exampleconnection}
\end{figure}

\begin{figure}[!htpb]
\scalebox{0.94}{\includegraphics{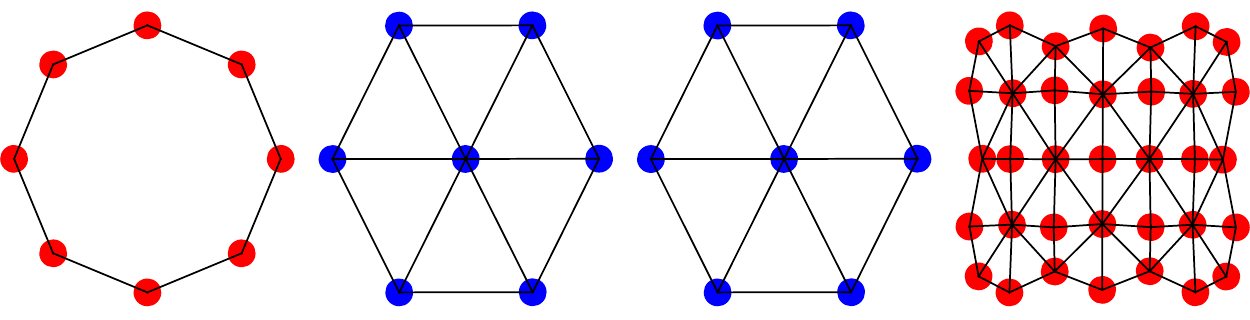}}
\caption{
The Barycentric refinement $G_1$ of $G$.
}
\label{examplebarycentric}
\end{figure}

\paragraph{}
Figure~({\ref{S2S3}) shows $G=S^2 \times S^3$. The 2-sphere $S^2$ is implemented as the
Octahedron graph with $f$-vector $(6,12,8)$, the smallest 2-sphere, which already
Descartes has super-summed to $6-12+8=2$. The number of cells is $26$.
The complex $S^3$ is the suspension of $S^2$. It has $f$-vector $(8,24,32,16)$ 
which super-sums to $\chi(S^3)=8-24+32-16=0$ as any 3-manifold does. The number of cells 
is $80$. The product has $26*80=2080$ cells
It is a 3-sphere, a cross polytop. The Poincar\'e polynomial of $G$ is 
$(1+x^2) (1+x^3) = 1+x^2+x^3+x^5$. Indeed, the Betti vector of this 5-manifold
is $(1,0,1,1,0,1)$. 

\paragraph{}
We see how useful the ring is. We did not have to build a 
triangulation of the 5-manifold as we had done in \cite{KnillKuenneth} 
where we defined the product to be $G_1$ in order to have a Whitney complex of a graph
$G_1$ with 2080 vertices and 51232 edges.
The Dirac and Hodge operator for $G=S^2 \times S^3$ is seen in Figure~({S2S3}).
They are both $2080 \times 2080$ matrices. As for any
5-dimensional complex, the Hodge operator has 6 blocks.
Blocks $H_1,H_3,H_4,H_6$ have a one-dimensional kernel.
McKean-Singer super symmetry shows that for the
union of the non-zero spectra $H_1,H_3,H_5$ is the
union of the non-zero spectra of $H_2,H_4,H_6$.

\paragraph{}
The ring element $G=C_4-2 K_3 + (L_2 \times L_3)$ in the strong ring
is a sum of three parts, where the first is a Whitney complex of a graph,
the second is $-2$ times a Whitney complex and the third is a product of
two Whitney complexes $L_2$ and $L_3$. We drew the weak Cartesian product to visualize $L_2 \times L_3$.
In reality, $L_2 \times L_3$ is not a simplicial complex. There are 6 additional
two dimensional cells present in the CW complex representing it,
one for each of the 6 holes present in the weak product.

\paragraph{}
Figure~(\ref{nonunique}) illustrates non-unique factorization in the strong ring
It is adapted from \cite{HammackImrichKlavzar} in weak ring.
The two ring elements  $G_1 \times G_2$ and $H_1 \times H_2$ are the
same. It is the decomposition
$(1+x+x^2)(1+x^3) =(1+x^2+x^4)(1+x)$, where $1$ is a point $K_1$
$x$ is an interval $K_2$. This gives the square $x^2 = K_2 \times K_2$,
the cube $x^3=K_2 \times K_2 \times K_2$ and hyper cube
$x^4=K_2 \times K_2 \times K_2 \times K_2$.

\begin{figure}[!htpb]
\scalebox{0.25}{\includegraphics{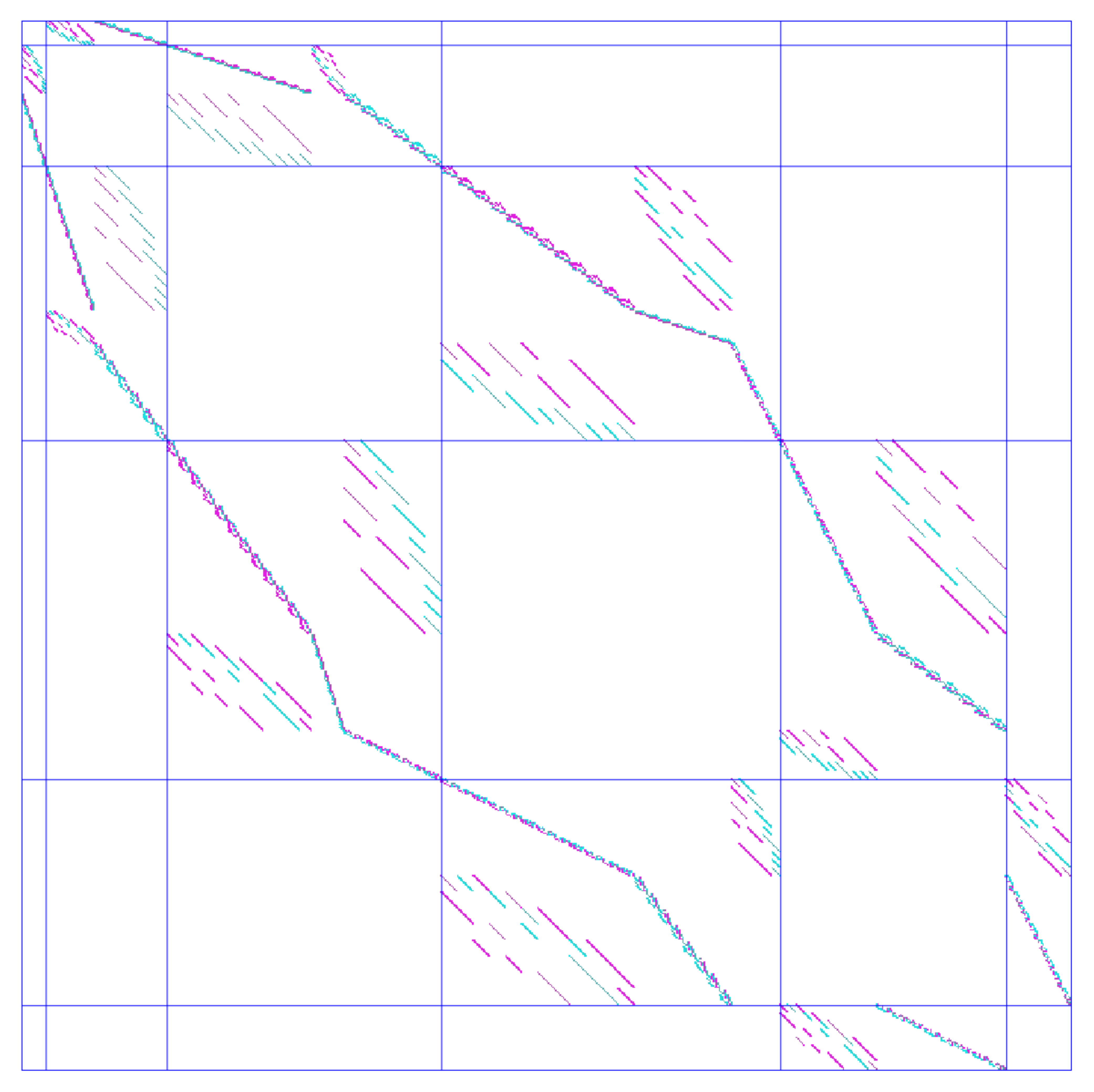}}
\scalebox{0.25}{\includegraphics{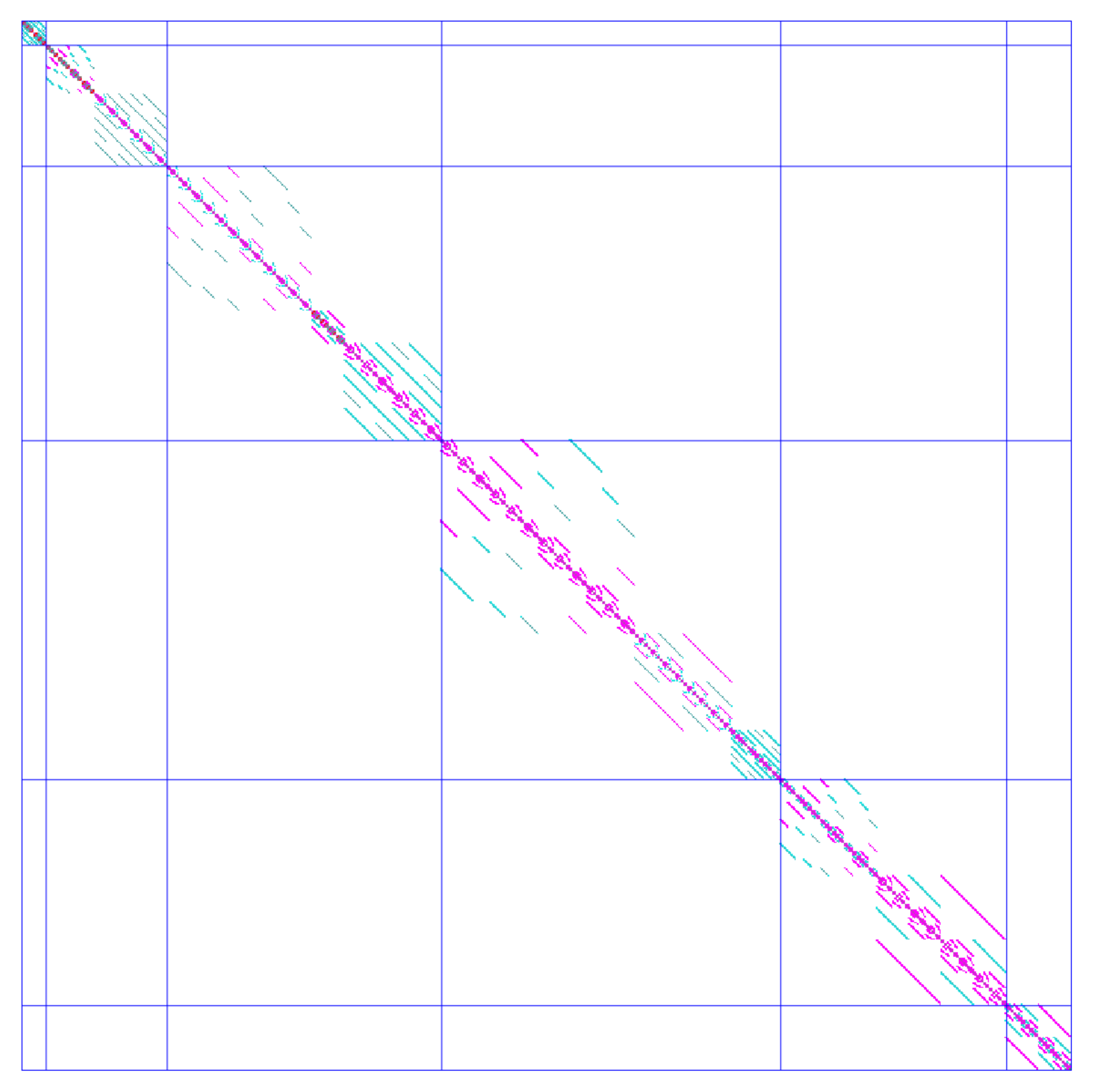}}
\caption{
The Dirac and Hodge operator for $G=S^2 \times S^3$. 
}
\label{S2S3}
\end{figure}

\begin{figure}[!htpb]
\scalebox{0.24}{\includegraphics{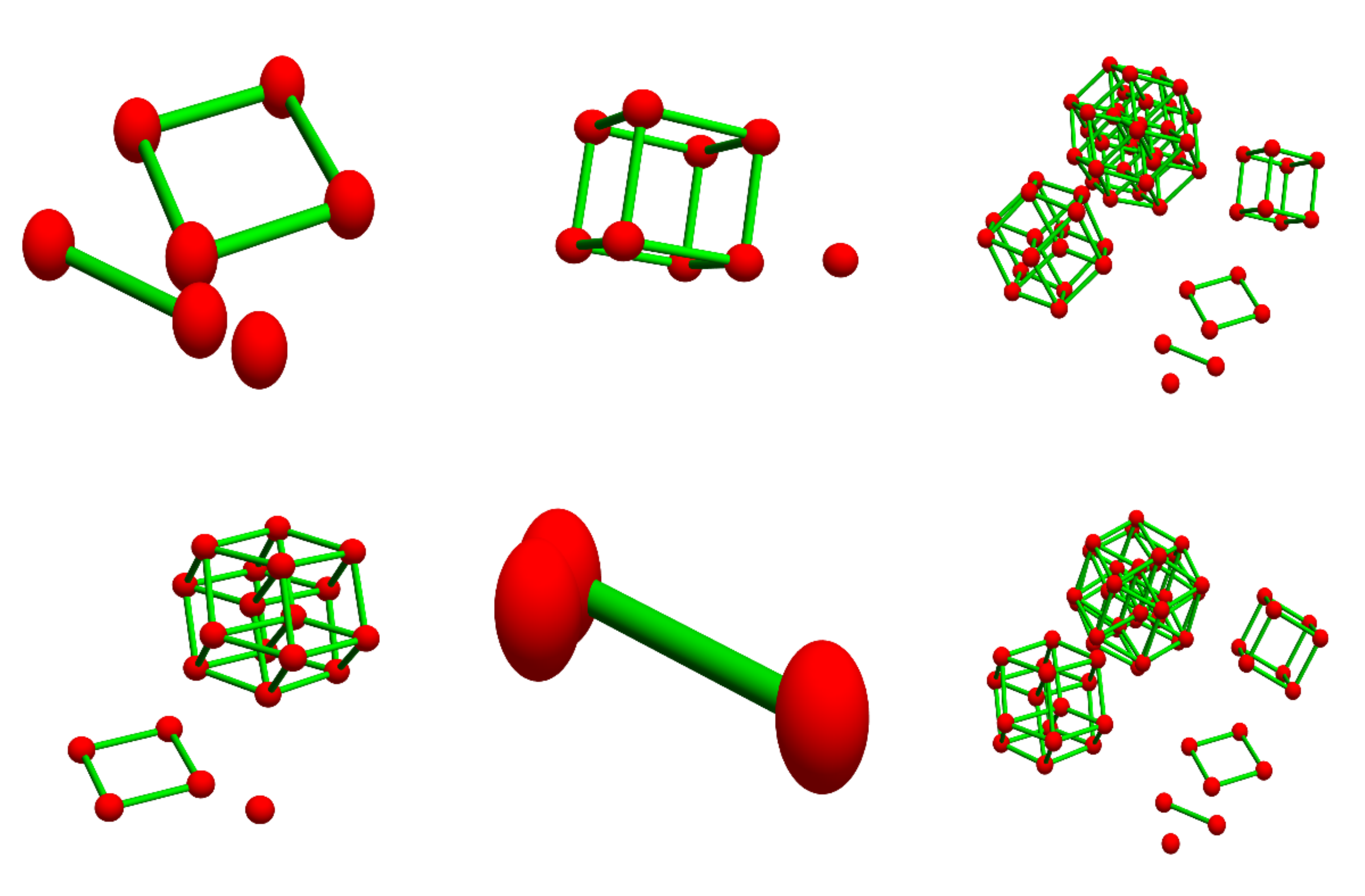}}
\caption{
Non-unique prime factorization. The products $AB$
and $CD$ of the two ring elements produces the same 
$G$. This is only possible if $G$ is not connected.
}
\label{nonunique}
\end{figure}

\begin{figure}[!htpb]
\scalebox{0.19}{\includegraphics{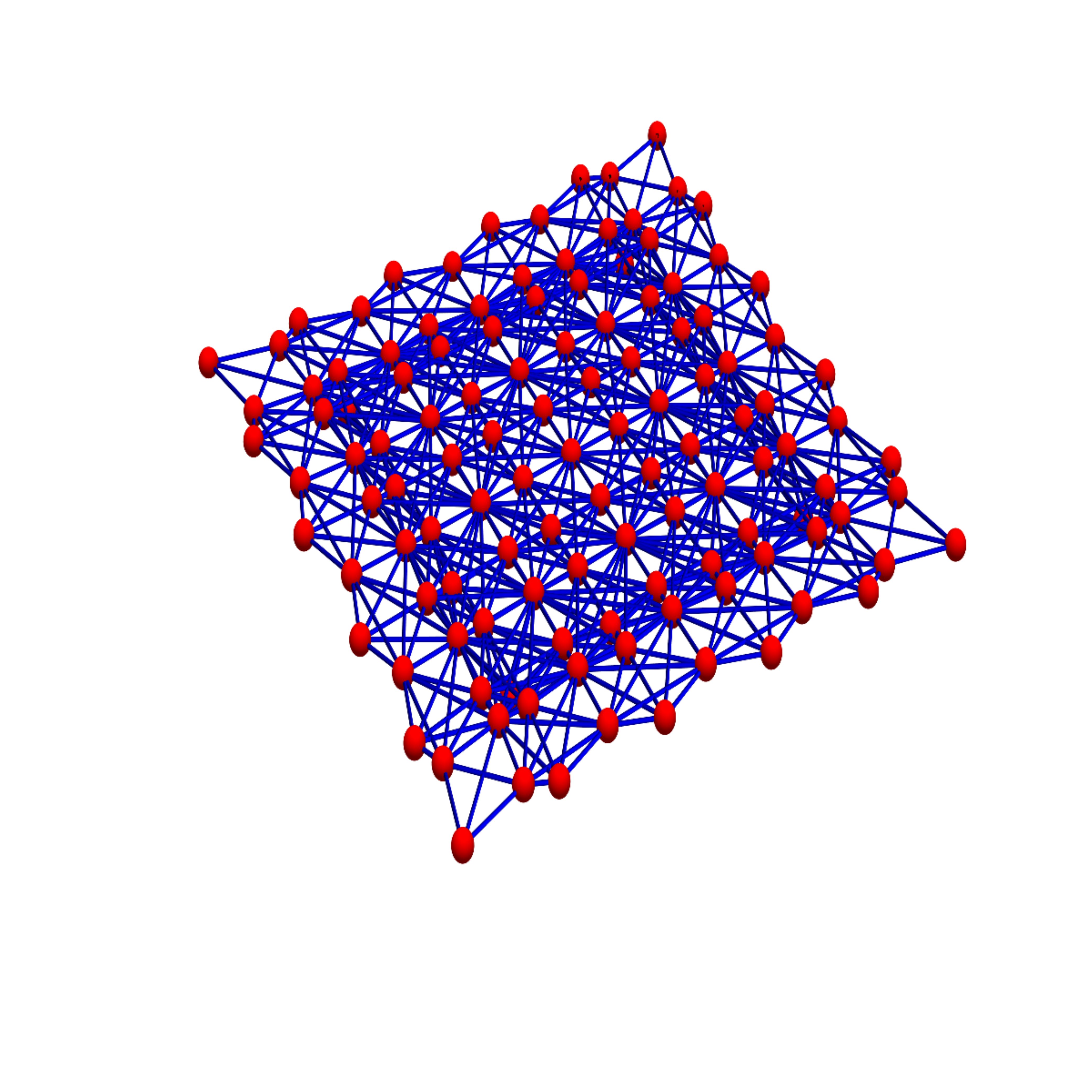}}
\caption{
The connection graph of $G=L_5 \times L_5$ where $L_n$ is the linear 
graph of length $n$ is part of the connection graph of the discrete
lattice $Z^1 \times Z^1$. The adjacency matrix $A$ of the graph $G$
seen here has the property that $L=1+A$ is the connection graph which is
invertible. 
}
\end{figure}

\begin{figure}[!htpb]
\scalebox{0.24}{\includegraphics{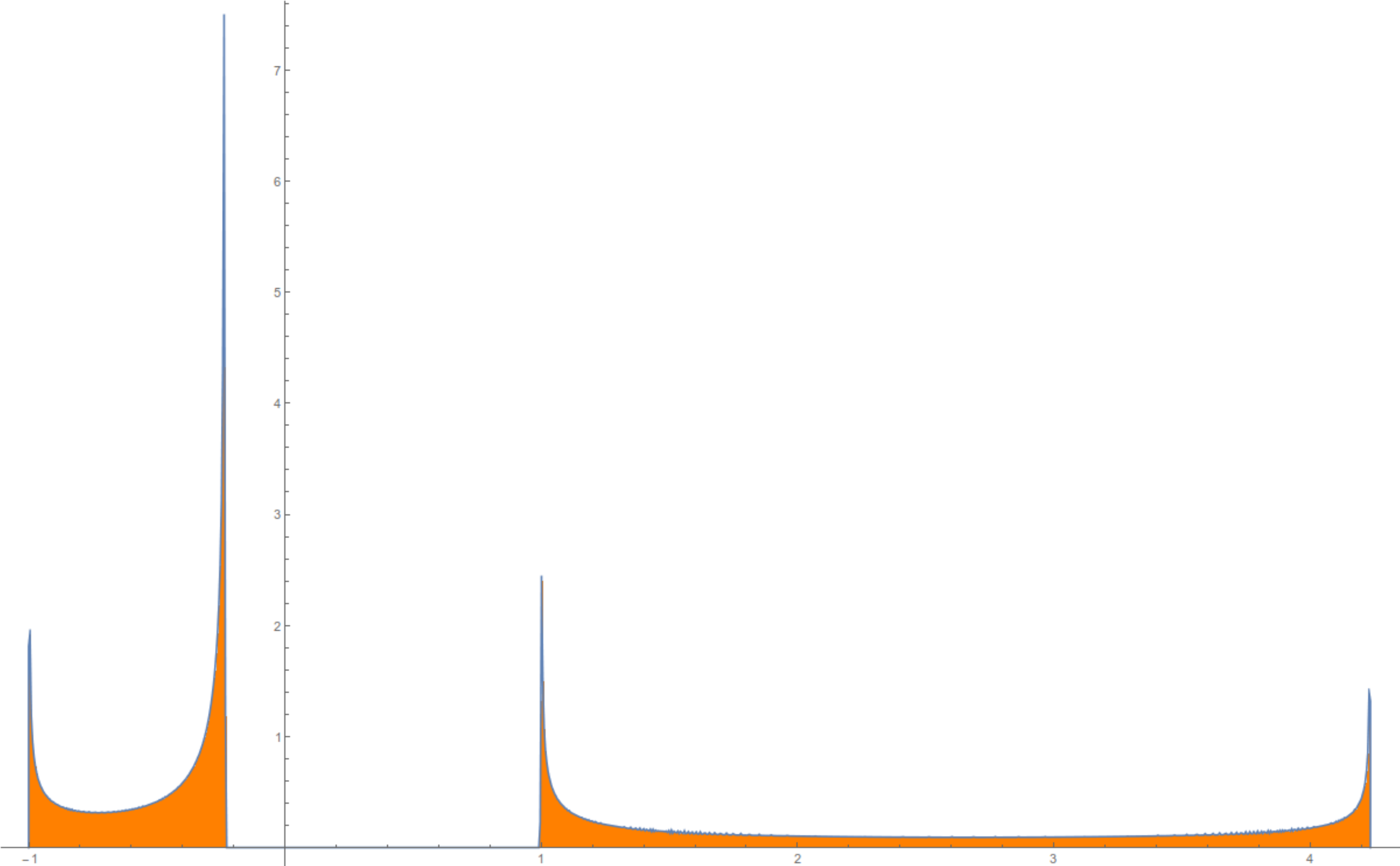}}
\caption{
The density of states of the connection Laplacian of $\ZZ$ has a mass gap at $0$.
We actually computed the eigenvalues of the connection Laplacian of $C_{10000}$ which has
a density of states close to the density of states of the 
connection Laplacian of $\ZZ$. The mass gap contains $[-1/5,1/5]$. 
}
\end{figure}

\begin{figure}[!htpb]
\scalebox{0.24}{\includegraphics{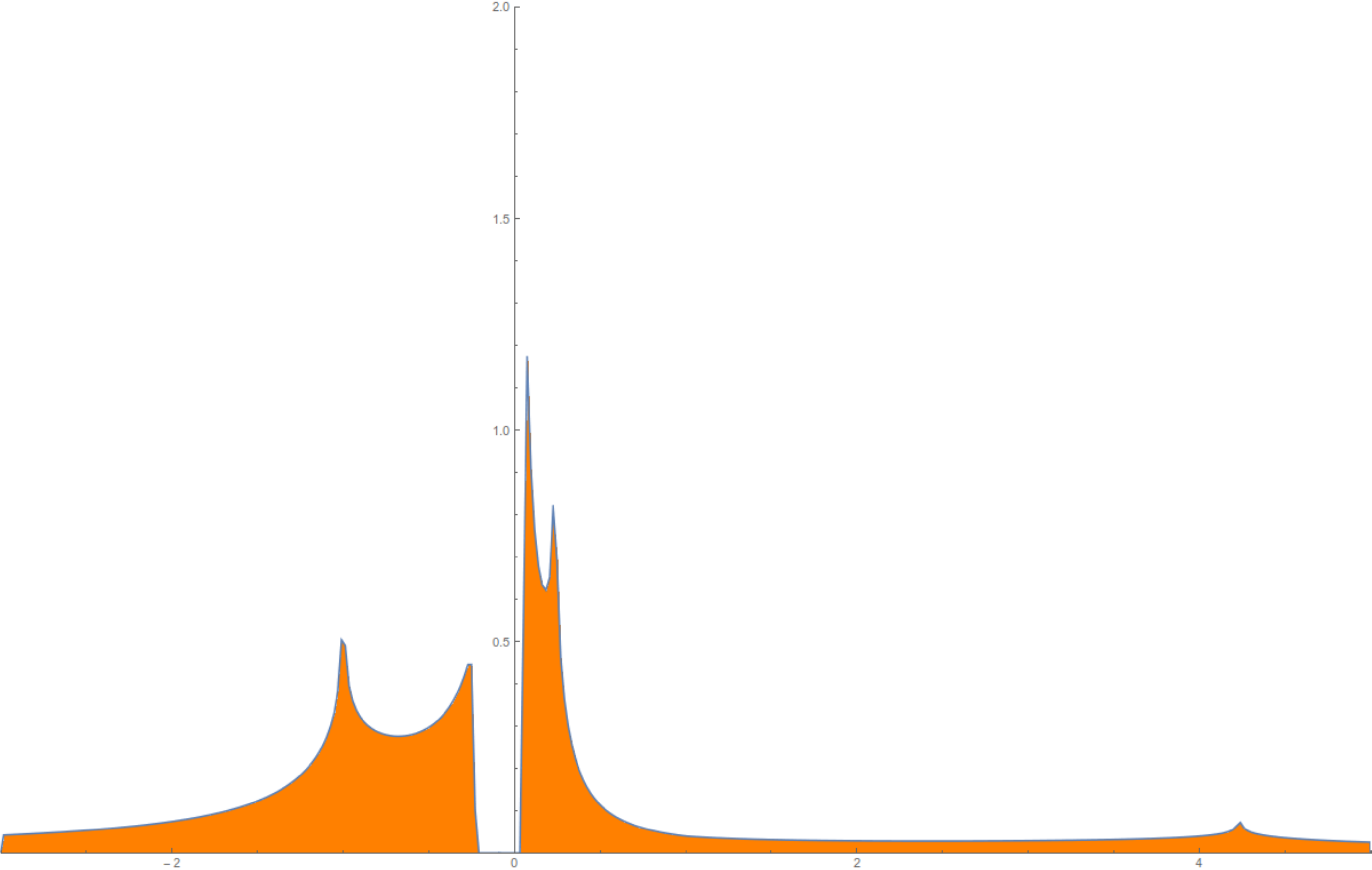}}
\caption{
Part of the density of states of $\ZZ^2$ for finite dimensional approximations like
$G=C_n \times C_n = C_{1000} \times C_{1000}=G_1 \times G_2$.
The eigenvalues of $L(G)$ are the products $\lambda_i \lambda_j$ 
of the eigenvalues  $\lambda_i$ of $L(G_i)$. 
The mass gap contains $[-1/25,1/25]$ which is independent of $n$. 
}
\end{figure}

\section*{Cartesian closed category}

\paragraph{}
The goal of this appendix is to see the strong ring of
simplicial complexes as a {\bf cartesian closed category} and to ask
whether it is a {\bf topos}. As we have already finite products, the first
requires to show the existence of exponentials.
Cartesian closed categories are important in computer science as they have 
{\bf simply typed lambda calculus} as language. Also here, we are close to computer science as 
we deal with a category of objects which (if they are small enough) can be realized 
in a computer. The elements can be represented as polynomials in a 
ring for example. We deal with a combinatorial category which can be explored without 
the need of finite dimensional approximations. It is part of combinatorics as all objects
are finite. 

\paragraph{}
In order to realize the ring as a category we need to define the
{\bf morphisms}, identifying an initial and terminal object (here $0=\emptyset$ and $1=K_1$) 
and show that {\bf currying} works: there is an exponential object $K^H$ in the ring
such that the set of morphisms $C(G \times H,K)$ from $G \times H$ to $K$ 
corresponds to the morphisms $C(G,K^H)$ via a {\bf Curry bijection} seeing a graph $z=f(x,y)$ 
of a function of two variables as a graph of the function 
$x \to g_x(y)=f(x,y)$ from $G$ to functions from $H$ to $K$. 

\paragraph{}
The existence of a product does not guarantee that a category is
cartesian closed. Topological spaces or smooth manifolds are not
cartesian closed but compactly generated Hausdorff spaces are. 
In our case, we are close to the category of finite sets which is
cartesian closed. Like for finite sets we are close to computer science as
procedures in computer programming languages are using the Curry bijection. 
Since the object $K^H$ is in general very large, it is as for now 
more of theoretical interest.

\paragraph{}
The strong ring $R$ resembles much the category of sets but there are negative
elements in $R$. We can look at the category of {\bf finite signed sets} which is
the subcategory of zero dimensional signed simplicial complexes. Also this
is a ring. It is isomorphic to $\ZZ$ as a ring but the set of morphisms produces
a category which has more structure than the ring $\ZZ$. 
This category of signed $0$-dimensional simplicial complexes is Cartesian closed in the
same way than the category of sets is. It is illustrative to see the exponential 
element $2^G$ is th set of all subsets of $G$. But this shows how exponential elements 
can become large. 

\paragraph{}
A simplicial complex $G$ as a finite set of non-empty sets, which is
closed under the operation of taking non-empty subsets. 
[The usual definition is to looking at the {\bf base set} $V=\bigcup_{x \in G} x$
and insisting that $G$ is a set of subset of $V$ with the property that if $y \subset x$
and $x \in G$ then $y \in G$ and also asking that $\{v\} \in G$ for every $v \in V$.
Obviously the first given {\bf point-free definition} is equivalent. ]
There is more structure than just the set of sets as the elements in $G$ are 
partially ordered. A morphism between two simplicial complexes is not just a 
map between the sets but an {\bf order preserving map}. This implies that the simplices
are mapped into each other. We could rephrase this that a morphism induces a graph 
homomorphism between the barycentric refinements $G_1$ and $H_1$ 
but there are more graph homomorphisms in general on $G_1$.  

\paragraph{}
The class $\C$ of simplicial complexes is a category for which the 
objects are the simplicial complexes and the morphisms are 
{\bf simplicial functions}, functions which preserve simplices. In
the point-free definition this means to look at functions 
from $G$ to $H$ which preserve the partial order. 
In order to be close to the definition of continuous functions 
(the morphisms in topological spaces) or measurable functions
(the morphisms in measure spaces) one could ask that $f^{-1}(A)$ 
is a simplicial complex for every simplicial complex. As $f^{-1}(A)$ can be
the empty complex this is fine. [If for some set $y \in H$, the inverse $f^{-1}(x)$ 
can be empty would not be good since simplicial complexes never contain the empty set.
This is fine for the empty complex, which does not contain the empty set neither.
But if we look at $x$ as a complex by itself, then $f^{-1}(x)$ can be the empty complex.
It is in general important to distinguish the elements $x$ in the simplicial complex from the
subcomplex $x$ it represents, evenso this is often not done. ]

\paragraph{}
The category of simplicial complexes is close to the category of finite
sets as the morphisms are just a subclass of all functions. There is an other
essential difference: the product $G \times H$ of two simplicial complexes is 
{\bf not} a simplicial complex any more in general, while the product $G \times H$
as sets is a set. This is also different from 
{\bf geometric realizations} of simplicial complexes (called ``polyhedra" 
in algebraic topology), where the product is a simplicial complex, the geometric
realization of the Barycentric refinement of $G \times H$ will do. 

\paragraph{}
[To compare with topologies O, where continuous maps have the property that $f^{-1}(x) \in O$
for every $x \in O$, morphisms of simplicial complexes have the property $f(x) \in H$ for $x \in G$.
But only surjective morphisms also have the property that $f^{-1}(y) \in G$ for every $y \in H$, the
reason being that $f^{-1}(x)$ can be empty. 
A continuous map on topological spaces which is also open has the property that $f$ and $f^{-1}$ preserve
the topology. A constant map for example is in general not open. We see that not only the object
of simplicial complex is simpler but also that the morphisms are simpler.]
It is better to therefore of a morphism $f$ between simplicial complexes
as a map for which both $f$ and $f^{-1}$ preserve sub simplicial complexes.

\paragraph{}
In order to work within the class of simplicial complexes (actually the special case of 
Whitney complexes of graphs), we had looked in \cite{KnillKuenneth} at the Barycentric
refinements of Cartesian products and called this the Cartesian product.
We had to live however with the consequence that
the product $(G,H) \to (G \times H)_1$ is {\bf not associative}: already $(G \times K_1)_1 = G_1$
is the Barycentric refinement of $G$. While the geometric realization of the Barycentric
refinement $G_1$ is topologically equivalent to $G$, there is a problem with products
as in the topological realization $|K_2 \times K_2| = |K_4|$ meaning that the arithmetic
is not compatible. The geometric realization destroys the arithmetic. In the strong ring
$K_4$ {\bf is a multiplicative prime}, in the geometric realization, it is not; it 
decays as $K_2 \times K_2$. 

\paragraph{}
Having enlarged the category to the strong ring, we have not only to deal with morphisms for
simplicial complexes, we also have to say what the morphisms in the ring are. 
The definition is recursive with respect to the {\bf degree} of a ring element, 
where the degree is the degree in the Stanley-Reisner 
polynomial representation. A map $G \to H$ is a morphism, if it is a 
morphism of simplicial complexes if $G,H$ are simplicial complexes and if 
for every pair $G,H$ in the ring, there is a pair $A,B$ in the ring and morphisms 
$g:G \to A, h:H \to B$ such that $f(G \times H)=g(G) \times h(H)$. 
[By the way, the degree of a monomial in the Stanley-Reisner representation $f_G$ only relates 
to the dimension if we $G$ is prime, that is if $G$ is a simplicial complex. In an product $A \times B$,
the degree of a monomial is $c(A) + c(B)$, where $c$ is the clique number. It is the clique 
number which is additive and not the dimension. The monomial $abcd$ in $(a+b+ab) (c+d+cd)$ 
for example belongs to a two-dimensional cell.  ]

\paragraph{}
The strong ring $S$ is a ring and a category. But it is itself an element in the category 
of rings. The image of the map $\phi: G \to G'$ is a subring $R$ of the Sabidussi ring of all 
graphs. The map $\phi$ is a ring isomorphism. If we think of $S$ as a category, then $R$ can
be thought so too and $\phi$ is now a functor. 
This is nothing strange. Category is a universal language where objects of categories
can be categories themselves. A directed graph for example is a category, where the objects
are the vertices and the morphisms are the directed edges.

\paragraph{}
The strong ring is also a {\bf cartesian monodidal category}, a category with a notion of 
tensor product. The unit is the unit in the ring. 
It is also a {\bf finitely complete category}, which is a category in which {\bf pullbacks} exist:
given any two ring elements $G,H$ and two morphisms $g:G \to K, h: H \to K$, there 
is a subcomplex $K$ of $G \times H$ such that for 
all $(x,y) \in K$, the equation $g(x)=h(y)$ holds. The subcomplex $K$ is called a 
pullback.

\paragraph{}
The strong ring appears also to be a {\bf topos} but we have not yet checked that.
A topos is a cartesian closed category with a 
sub-object classifier. Examples of topoi are sets or the G-dynamical
systems for a group G or the category of sheaves on a topological space.
Topoi enjoy stability properties: the fundamental theorem of
topos theory tells that a topos is stable under slicing, i.e that it is
locally cartesian closed. 

\vfill

\pagebreak

\bibliographystyle{plain}

\end{document}